\documentclass[11pt]{article}
\usepackage[english]{babel}

\usepackage[a4paper,margin=1in]{geometry}
\usepackage{amsmath,amssymb,amsthm,mathtools,bm,mathrsfs}
\usepackage{enumitem}
\usepackage{microtype}
\usepackage{hyperref}

\usepackage[english]{babel}
\usepackage{geometry}
\usepackage{pdfpages}
\usepackage{amsmath}
\usepackage{amssymb}
\usepackage{amsthm}
\usepackage{graphicx}
\usepackage{float}
\usepackage{array}
\usepackage{xcolor}
\usepackage{listings}
\usepackage{tikz}
\usetikzlibrary{arrows.meta, calc, positioning}

\definecolor{keywordcolor}{rgb}{0.7,0.1,0.1}
\definecolor{commentcolor}{rgb}{0.4,0.4,0.4}
\definecolor{symbolcolor}{rgb}{0.0,0.1,0.6}
\definecolor{sortcolor}{rgb}{0.1,0.5,0.1}

\usepackage{lstlean}

\lstdefinestyle{LeanCode}{
  language=lean,
  basicstyle=\ttfamily\fontsize{8pt}{8pt}\selectfont,
  keywordstyle=[1]{\color{keywordcolor}},
  keywordstyle=[2]{\color{sortcolor}},
  commentstyle=\itshape\color{commentcolor},
  stringstyle=\ttfamily,
  breaklines=true,
  breakatwhitespace=false,
  columns=fullflexible,
  keepspaces=true,
  showstringspaces=false,
  tabsize=2,
  extendedchars=false,
  captionpos=b
}

\usepackage{newunicodechar}
\newunicodechar{⦃}{\ensuremath{\{\!\{}}
\newunicodechar{⦄}{\ensuremath{\}\!\}}}
\newunicodechar{⋂}{\ensuremath{\bigcap}}
\newunicodechar{ₗ}{\ensuremath{_{l}}}
\newunicodechar{ᵢ}{\ensuremath{_{i}}}
\newunicodechar{ℝ}{\ensuremath{\mathbb{R}}}
\newunicodechar{ₜ}{\ensuremath{_{t}}}

\newtheorem{theorem}{Theorem}
\newtheorem{proposition}[theorem]{Proposition}
\newtheorem{corollary}[theorem]{Corollary}

\theoremstyle{definition}
\newtheorem{definition}{Definition}

\newcommand{\R}{\mathbb{R}}
\newcommand{\Sone}{\mathbb{S}^{1}}
\newcommand{\Int}{\operatorname{int}}
\newcommand{\diam}{\operatorname{diam}}
\newcommand{\dist}{\operatorname{dist}}
\newcommand{\conv}{\operatorname{conv}}
\newcommand{\Area}{\operatorname{Area}}

\newcommand{\SO}{\operatorname{SO}}
\newcommand{\OO}{\operatorname{O}}
\newcommand{\op}{\mathrm{op}}

\title{Closed escape path of smallest diameter in forest: dual formulation of Lebesgue's universal covering problem}
\author{Zhipeng Deng}
\date{} 

\begin{document}
\maketitle

\begin{abstract}
Lebesgue’s universal covering problem asks for the minimum area convex planar region capable of containing a congruent copy of every planar set of diameter at most one. In this paper, we develop an exact dual formulation of this problem through a minimum diameter analogue of Bellman’s lost-in-a-forest problem. For a compact convex forest $F$, we define the critical escape diameter $D(F)$ as the minimum diameter of a closed path whose trace cannot be placed, under any rigid motion, entirely in the interior of $F$. We prove attainment of this minimum and establish a diameter cover/escape duality showing that the normalized body $D(F)^{-1}F$ is a Lebesgue universal cover. Consequently, the Lebesgue universal covering constant admits the exact representation
\[
\mathcal L=\inf_F\frac{\operatorname{Area}(F)}{D(F)^2},
\]
where the infimum ranges over compact convex planar bodies with nonempty interior. By reversing the rigid motion, we further characterize escape as intersection of a fixed path with every oppositely transformed boundary of $F$, and derive equivalent continuous curve, convex body, and support function optimization formulations. To make the infinite dimensional problem computationally tractable while retaining rigorous control of approximation error, we discretize the compact configuration space by an $\eta_m$-net and formulate the resulting problem as a minimum diameter traveling salesman problem with neighborhoods; for polygonal forests, an exact mixed-integer second-order cone formulation is obtained. We prove the quantitative certification, which yields convergent, rigorously certified universal cover bounds. The framework replaces finite tests of prescribed constant width shapes by a unified optimization over the full configuration space and extends naturally to other congruent and translative universal cover problems.
\end{abstract}

\\

\noindent \textbf{Classification}

Metric Geometry (math.MG), Optimization and Control (math.OC)

49K30, 49Q10, 52A40

\tableofcontents

\section{Introduction}

Lebesgue's universal covering problem is an unsolved problem in geometry that asks for the convex shape of smallest area that can cover every planar set of diameter one. A shape covers a set if it contains a congruent subset, which may be rotated and translated \cite{Brass2005}. 

The problem was posed by Henri Lebesgue in a letter to Pál. It was published by Pál in 1920 \cite{Pál1920}. As planar set of diameter one can be embedded inside a curve of constant unit width, much of the literature revolves around coverage of curves of constant unit width, such as circles, Reuleaux triangles, and Reuleaux polygons \cite{Brass2005-1} \cite{Gibbs2014}. 

For upper bound of minimum area, it has historically been a process of geometric refinement, removing tiny slivers of area from known covering shapes. Pál \cite{Pál1920} proved that a regular hexagon circumscribed around a circle of unit diameter is a universal cover. Pál then demonstrated that two corners of this hexagon could be truncated without losing the covering property, establishing the first rigorous upper bound. Then Sprague \cite{Sprague1936} proved that a small portion of a third corner could be removed. After many years, Hansen \cite{Hansen1992} successfully reduced the upper bound further by discovering two additional regions that could be safely removed from Sprague's cover. The search was revitalized by \cite{Gibbs2014} \cite{Baez2015}, who utilized computerized search methods to shave off even more area. Pushing these methods to limits, Gibbs \cite{Gibbs2018} achieved further infinitesimal reductions. The best known upper bound currently is 0.84409. 

As for lower bounds of minimum area, they require proving that no convex shape below a certain area can simultaneously cover all unit diameter sets. This is typically achieved by selecting a finite family of diameter one shapes and calculating the minimum convex area required to cover just that family. Brass \cite{Brass2005-1} established a foundational lower bound by analyzing the optimal overlapping alignments of three specific shapes: circle, Reuleaux triangle, and Reuleaux pentagon. More recently, Xie \cite{Xie2026} reproduced this method by providing a computationally certified lower bound. Chen \cite{Chen2026} improved the bounds on the minimum volume of a universal cover in $\mathbb{R}^3$.

However, these previous approaches were mostly based on geometric methods. They yield only constructed upper and lower bounds, which were limited to considering just a few specific polygons of constant width. The difficulty lies in the fact that there are infinitely many shapes with a diameter of one, and the space involving translation and rotation are infinite. Progress of this problem has been slow, and no general method exists.

This paper first defines Lebesgue's universal covering problem, one kind of universal cover problem \cite{Brass2005}. Then, by exploiting the duality between Bellman’s lost-in-s-forest problem and Moser’s worm problem \cite{FinchWetzel2004, Deng2024, Deng2026, Deng2026-2, Deng2026-3}, our contribution is to formally construct dual problem of Lebesgue's universal covering problem. Specifically, it is finding a closed escape path with the shortest diameter in a forest, where the solution to this dual problem provides an upper bound for Lebesgue's problem. Furthermore, we present the formulation for solving this dual problem. It places no specific and predefined requirements on the shape of constant unit width curves. Thereby, we enable the iterative refinement of the dual problem to optimize the upper bound for Lebesgue's universal covering problem.

\section{Proof of concept}

Figure 1 shows the proof of concept. A universal cover problem and Bellman's lost-in-a-forest problem are two views of the same rigid-motion obstruction.  In the covering view, a fixed region is moved relative to every member of a prescribed family until the member lies inside the region.  In the escape view, the path is fixed and the forest boundary is translated and rotated in the opposite direction.  The path succeeds precisely when it intersects every transformed boundary.

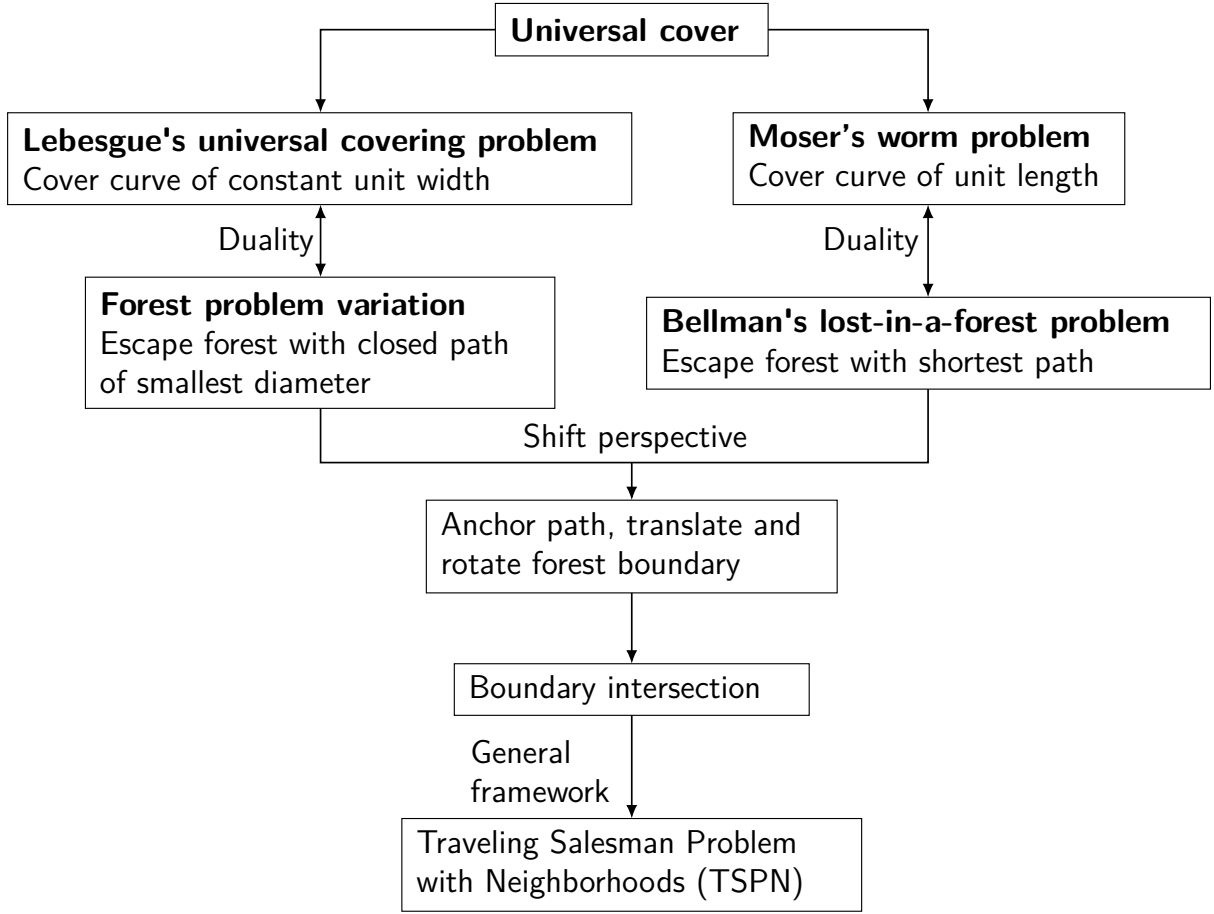
\begin{figure}[htbp]
\centering
\resizebox{\linewidth}{!}{%
\begin{tikzpicture}[
    x=0.011cm, y=-0.011cm,          
    font=\sffamily\fontsize{11.8}{14.1}\selectfont,
    box/.style={draw, thin},
    txt/.style={anchor=west, align=left, inner sep=0pt},
    arr/.style={semithick, -{Latex[length=1.7mm, width=1.4mm]}},
    dbl/.style={semithick, {Latex[length=1.7mm, width=1.4mm]}-{Latex[length=1.7mm, width=1.4mm]}},
    line join=miter
]
\draw[box] (570,10) rectangle (882,70);
\node[txt] at (586,40) {\textbf{Universal cover}};

\draw[box] (12,135) rectangle (726,240);
\node[txt] at (28,187.5) {\textbf{Lebesgue\textquotesingle s universal covering problem}\\
                         Cover curve of constant unit width};

\draw[box] (842,135) rectangle (1290,240);
\node[txt] at (858,187.5) {\textbf{Moser's worm problem}\\
                          Cover curve of unit length};

\draw[box] (101,323) rectangle (638,473);
\node[txt] at (117,398) {\textbf{Forest problem variation}\\
                        Escape forest with closed path\\
                        of smallest diameter};

\draw[box] (743,346) rectangle (1388,451);
\node[txt] at (759,398.5) {\textbf{Bellman\textquotesingle s lost-in-a-forest problem}\\
                          Escape forest with shortest path};

\draw[box] (491,578) rectangle (961,684);
\node[txt] at (507,631) {Anchor path, translate and\\
                        rotate forest boundary};

\draw[box] (523,765) rectangle (929,824);
\node[txt] at (539,794.5) {Boundary intersection};

\draw[box] (463,942) rectangle (989,1048);
\node[txt] at (479,995) {Traveling Salesman Problem\\
                        with Neighborhoods (TSPN)};

\draw[arr] (570,40) -| (370,135);
\draw[arr] (882,40) -| (1065,135);

\draw[dbl] (370,240) -- (370,323);
\node[anchor=east, inner sep=0pt] at (363,283) {Duality};
\draw[dbl] (1065,240) -- (1065,346);
\node[anchor=east, inner sep=0pt] at (1055,283) {Duality};

\draw[semithick] (370,473) -- (370,535) -- (1065,535) -- (1065,451);
\node[anchor=south, inner sep=0pt] at (730,525) {Shift perspective};
\draw[arr] (726,535) -- (726,578);

\draw[arr] (726,684) -- (726,765);

\draw[arr] (726,824) -- (726,942);
\node[txt] at (541,887.5) {General\\ framework};
\end{tikzpicture}%
}
\caption{Proof of concept}
\label{fig:proof-of-concept}
\end{figure}

For Lebesgue's problem the objects to be covered are all planar sets of diameter at most one.  The planar completion theorem reduces this family to convex bodies of constant width one.  Their boundaries are closed convex curves of diameter one.  This suggests replacing the usual length objective in Bellman's problem by the diameter of a closed escape path.  For a fixed convex forest $F$, define the critical escape diameter to be the least diameter of a closed path no rigid copy of which lies in $\Int F$.  Scaling $F$ by the reciprocal of this critical diameter produces a universal cover for all sets of diameter one.  Thus the area ratio
\[
    \frac{\Area(F)}{D(F)^2}
\]
is an admissible upper bound for Lebesgue's universal covering constant.

The computational viewpoint follows the same reversal of motion.  Anchor the path in a canonical coordinate system.  For every unknown starting point and orientation, translate and rotate the forest boundary oppositely.  The anchored path must intersect every resulting boundary.  Sampling the compact configuration space gives finitely many boundary neighborhoods.  The sampled problem is therefore a traveling salesman with neighborhoods problem with a minimum-diameter objective.  Refining the configuration-space mesh yields certified lower and upper bounds converging to the continuous critical diameter.

\section{Formulation of Lebesgue's universal covering problem and dual problem}

\subsection{Universal covers and reduction to constant width curve}

Let $\mathcal Q=\OO(2)$ and let a planar rigid motion be written as
\[
    g_{a,Q}(x)=a+Qx,
    \qquad a\in\R^2,
    \quad Q\in\mathcal Q.
\]
Using $\SO(2)$ instead of $\OO(2)$ gives the orientation-preserving version; every argument below is unchanged.  The full orthogonal group is used because congruence in Lebesgue's problem conventionally permits reflection.

\begin{definition}[Cover and universal cover]
A set $F\subset\R^2$ \emph{covers} a compact set $S\subset\R^2$ if
\[
    g_{a,Q}(S)\subset F
\]
for some rigid motion $g_{a,Q}$.  A \emph{Lebesgue universal cover} is a compact convex set $F$ with nonempty interior that covers every compact planar set $S$ satisfying $\diam S\leq 1$.  The Lebesgue universal covering constant is
\[
    \mathcal L
    :=
    \inf\left\{
        \Area(F):
        F\subset\R^2 \text{ is a Lebesgue universal cover}
    \right\}.
\]
\end{definition}

The use of $\diam S\leq1$ instead of $\diam S=1$ does not change the problem.  The completion theorem for planar sets states that every compact set of diameter at most one is contained in a compact convex body of constant width one; see \cite{Vrecica1981,BaezBagdasaryanGibbs2015}.  Recall that a convex body $W$ has constant width one if its support function satisfies
\[
    h_W(u)+h_W(-u)=1,
    \qquad u\in\Sone.
\]
Such a body is closed and convex, and its boundary is a simple closed convex curve of diameter one.

\begin{theorem}[Constant-width reduction]
For a compact convex set $F$ with nonempty interior, the following are equivalent:
\begin{enumerate}[label=\textup{(\roman*)}]
    \item $F$ is a Lebesgue universal cover;
    \item $F$ covers every planar convex body of constant width one;
    \item $F$ covers the boundary of every planar convex body of constant width one.
\end{enumerate}
\end{theorem}

\begin{proof}
The implication \textup{(i)}$\Rightarrow$\textup{(ii)} is immediate because constant-width bodies have diameter one.  For \textup{(ii)}$\Rightarrow$\textup{(i)}, let $S$ be compact with $\diam S\leq1$.  By the planar completion theorem, there exists a convex body $W$ of constant width one with $S\subset W$.  A rigid copy of $W$ contained in $F$ contains the corresponding rigid copy of $S$.

The implication \textup{(ii)}$\Rightarrow$\textup{(iii)} is immediate.  Conversely, if a rigid copy of $\partial W$ lies in the convex set $F$, then
\[
    g_{a,Q}(W)
    =
    \conv\bigl(g_{a,Q}(\partial W)\bigr)
    \subset F,
\]
so \textup{(iii)} implies \textup{(ii)}.
\end{proof}

The disk, the Reuleaux triangle, and Reuleaux polygons are important members of the constant-width class, but checking only a finite collection cannot establish universality.  The theorem requires all constant-width-one bodies, equivalently all of their boundary curves.

\subsection{Recall: Moser’s worm problem and Bellman's forest problem}

Let $\mathscr C_{\ell}$ denote the family of rectifiable planar arcs of length at most $\ell$.  For a compact convex forest $F$, define its interior worm-covering radius by
\[
    \rho_{\mathrm{len}}(F)
    :=
    \sup\left\{
        \ell\geq0:
        \text{every }\gamma\in\mathscr C_{\ell}
        \text{ has a rigid copy in }\Int F
    \right\}.
\]
Define escape length of Bellman's forest problem by
\[
    \beta_{\mathrm{len}}(F)
    :=
    \inf\left\{
        \operatorname{len}(\gamma):
        \text{no rigid copy of }\gamma
        \text{ is contained in }\Int F
    \right\}.
\]
Equivalently, a hiker follows a fixed canonical path while the unknown starting point and initial orientation range over all possibilities in $F$; the path is an escape path precisely when every such physical realization meets $\partial F$.  Finch and Wetzel proved the corresponding threshold identity for bounded forests \cite[Theorem~3]{FinchWetzel2004}.

\begin{proposition}[Length-cover/escape duality with endpoint closure]
\label{prop:length-cover-escape-duality}
Let $F\subset\mathbb R^2$ be compact with nonempty interior.  For a
rectifiable planar arc $\gamma$, write $\Gamma_\gamma$ for its trace, and write
$\Gamma_\gamma\prec F$ if some Euclidean rigid motion maps $\Gamma_\gamma$
into $\operatorname{Int}F$.  Define
\[
 \rho_{\mathrm{len}}^{\circ}(F)
 :=\sup\bigl\{\ell\ge 0:
        \Gamma_\gamma\prec F
        \text{ for every rectifiable }\gamma
        \text{ with }\operatorname{len}(\gamma)\le\ell\bigr\},
\]
and
\[
 \beta_{\mathrm{len}}(F)
 :=\inf\bigl\{\operatorname{len}(\gamma):
        \Gamma_\gamma\not\prec F\bigr\}.
\]
Then
\[
 \rho_{\mathrm{len}}^{\circ}(F)=\beta_{\mathrm{len}}(F)=:L.
\]
Moreover, $0<L<\infty$, every rectifiable arc of length strictly less than
$L$ has a rigid copy in $\operatorname{Int}F$, and every rectifiable arc of
length at most $L$ has a rigid copy in the closed set $F$.

If, in addition, $F$ is convex and $A=\operatorname{Area}(F)$, then
$L^{-1}F$ is a convex cover for every rectifiable planar arc of length at most
one and
\[
 \operatorname{Area}(L^{-1}F)=\frac{A}{L^2}.
\]
Consequently, $A/L^2$ is an upper bound for the convex version of Moser's
worm problem.
\end{proposition}

\begin{proof}
Put
\[
 \mathcal B_F
 :=\bigl\{\operatorname{len}(\gamma):
             \Gamma_\gamma\not\prec F\bigr\}.
\]
The set $\mathcal B_F$ is nonempty.  Indeed, since $F$ is compact, its
diameter $D$ is finite, and a line segment of length greater than $D$ cannot
have a rigid copy even in $F$.  It is bounded below by zero, so
$L:=\inf\mathcal B_F$ is a finite nonnegative real number.

We first prove the two strict threshold implications.  If $\ell<L$ and
$\operatorname{len}(\gamma)\le\ell$, then $\gamma$ cannot be non-fitting;
otherwise $\operatorname{len}(\gamma)\in\mathcal B_F$ and hence
\[
 L\le \operatorname{len}(\gamma)\le\ell<L,
\]
a contradiction.  Thus every arc of length at most $\ell$ fits in
$\operatorname{Int}F$, and therefore
\[
 \ell\le \rho_{\mathrm{len}}^{\circ}(F).
\]
Since this holds for every $\ell<L$, density of the real order implies
$L\le\rho_{\mathrm{len}}^{\circ}(F)$: otherwise one could choose
$\ell$ with
$\rho_{\mathrm{len}}^{\circ}(F)<\ell<L$, contradicting
$\ell\le\rho_{\mathrm{len}}^{\circ}(F)$.

Conversely, if $\ell>L$, the defining property of the infimum gives
$r\in\mathcal B_F$ with $r<\ell$.  Hence there is a rectifiable arc $\gamma$
such that
\[
 \Gamma_\gamma\not\prec F,
 \qquad
 \operatorname{len}(\gamma)=r<\ell.
\]
Thus the universal interior-covering property fails at every level
$\ell>L$.  Hence every member of the set whose supremum defines
$\rho_{\mathrm{len}}^{\circ}(F)$ is at most $L$, and therefore
$\rho_{\mathrm{len}}^{\circ}(F)\le L$.  This proves
\[
 \rho_{\mathrm{len}}^{\circ}(F)=L=\beta_{\mathrm{len}}(F).
\]

The positivity of $L$ follows from the nonempty interior of $F$.  Choose
$z\in\operatorname{Int}F$ and $r>0$ such that
$\overline B(z,r)\subset\operatorname{Int}F$.  If
$\operatorname{len}(\gamma)\le r$, translate the initial point $\gamma(0)$
to $z$.  For every parameter value $s$, the chord-length bound for a
rectifiable arc gives
\[
 \|\gamma(s)-\gamma(0)\|\le\operatorname{len}(\gamma)\le r.
\]
The translated trace is therefore contained in
$\overline B(z,r)\subset\operatorname{Int}F$.  Hence $L\ge r>0$.

It remains to justify closed containment at the endpoint.  Let $\gamma$ be
rectifiable with $\operatorname{len}(\gamma)\le L$, set $c=\gamma(0)$, and
choose $t_n\in(0,1)$ with $t_n\to1$.  Define the contraction
\[
 \gamma_n(s):=c+t_n(\gamma(s)-c).
\]
Then
\[
 \operatorname{len}(\gamma_n)
 =t_n\operatorname{len}(\gamma)
 \le t_nL<L.
\]
By the strict subcritical result, for every $n$ there are
$a_n\in\mathbb R^2$ and $Q_n\in O(2)$ such that
\[
 a_n+Q_n\Gamma_{\gamma_n}\subset\operatorname{Int}F.
\]
Set $b_n:=a_n+Q_nc$.  Since $c\in\Gamma_{\gamma_n}$, one has $b_n\in F$.
Compactness of $F$ and of $O(2)$ yields a subsequence, not relabelled, and
limits $b\in F$, $Q\in O(2)$ such that
\[
 b_n\to b,
 \qquad
 Q_n\to Q.
\]
For each $x\in\Gamma_\gamma$,
\[
 b_n+t_nQ_n(x-c)\in F
\]
and
\[
 b_n+t_nQ_n(x-c)\longrightarrow b+Q(x-c).
\]
Because $F$ is closed,
$b+Q(x-c)\in F$.  Therefore
\[
 b+Q(\Gamma_\gamma-c)\subset F,
\]
which is a rigid copy of $\Gamma_\gamma$ in $F$.

Finally, let $\eta$ be any rectifiable arc of length at most one.  The dilated
arc $L\eta$ has length at most $L$, so it has a rigid copy in $F$.  Dividing
that placement by $L$ gives a rigid copy of $\eta$ in $L^{-1}F$.  Homothety by
$L^{-1}$ preserves compactness and convexity and multiplies planar area by
$L^{-2}$.  Hence
\[
 \operatorname{Area}(L^{-1}F)=A/L^2,
\]
which proves the asserted worm-covering bound.
\end{proof}

It is noted that 
$\rho_{\mathrm{len}}^{\circ}(F)=\beta_{\mathrm{len}}(F)$ does not by itself
assert that every arc of length exactly $L$ fits in $\operatorname{Int}F$.
Indeed, strict interior coverage at $L$ holds if and only if the infimum in the
definition of $\beta_{\mathrm{len}}(F)$ is not attained by a non-fitting arc.
What compactness supplies at the threshold is closed containment in $F$.

\subsection{The smallest diameter escape forest problem}

A \emph{closed path} is a continuous map
\[
    \gamma:[0,1]\to\R^2,
    \qquad
    \gamma(0)=\gamma(1),
\]
and its trace is $\Gamma=\gamma([0,1])$.  Its diameter is
\[
    \diam\gamma
    :=
    \diam\Gamma
    =
    \max_{s,t\in[0,1]}
    \|\gamma(s)-\gamma(t)\|.
\]
Because the objective is translation invariant, the path may be normalized by $\gamma(0)=\gamma(1)=0$.

\begin{definition}[Closed diameter-escape path]
Let $F\subset\R^2$ be a compact convex set with nonempty interior.  A closed path $\gamma$ is a \emph{diameter-escape path} for $F$ if
\[
    a+Q\Gamma\not\subset\Int F
    \qquad
    \text{for every }a\in\R^2,
    \quad Q\in\mathcal Q.
\]
The critical escape diameter is
\begin{equation}
\label{eq:critical-diameter}
    D(F)
    :=
    \inf\left\{
        \diam\gamma:
        \gamma \text{ is a closed diameter-escape path for }F
    \right\}.
\end{equation}
\end{definition}

The term ``area of the boundary'' is not used: the relevant quantity is the planar area $\Area(F)$ of the forest region.

Define also the interior diameter-covering radius
\[
    \rho_{\mathrm{diam}}(F)
    :=
    \sup\left\{
        d\geq0:
        \begin{array}{l}
        \text{every compact }S\subset\R^2\text{ with }\diam S<d\\
        \text{has a rigid copy contained in }\Int F
        \end{array}
    \right\}.
\]

\begin{theorem}[Diameter-cover/escape duality and exact normalization]
\label{thm:diameter-duality}
Let $F\subset\R^2$ be compact, convex, and have nonempty interior.  Then:
\begin{enumerate}[label=\textup{(\alph*)}]
    \item the infimum in \eqref{eq:critical-diameter} is attained by a rectifiable closed path;
    \item
    $
        D(F)=\rho_{\mathrm{diam}}(F);
    $
    \item the scaled body
    $
        F^{\sharp}:=D(F)^{-1}F
    $
    is a Lebesgue universal cover and
    \[
        \Area(F^{\sharp})
        =
        \frac{\Area(F)}{D(F)^2};
    \]
    \item the Lebesgue constant admits the exact dual representation
    \begin{equation}
    \label{eq:global-dual}
        \mathcal L
        =
        \inf_{F}
        \frac{\Area(F)}{D(F)^2},
    \end{equation}
    where the infimum is over all compact convex planar bodies with nonempty interior.
\end{enumerate}
\end{theorem}

\begin{proof}
We first reduce the path problem to compact convex sets.  If $\Gamma$ is the trace of an escape path, let $K=\conv\Gamma$.  Then
$\diam K=\diam\Gamma.$

Moreover, if some rigid copy of $K$ were contained in $\Int F$, the corresponding copy of $\Gamma$ would also be contained there, contradicting escape.  Conversely, if a compact convex set $K$ has no rigid copy in $\Int F$, then its boundary is an escape trace when $K$ has nonempty interior.  Indeed, if $a+Q\partial K\subset\Int F$, convexity gives
\[
    a+QK
    =
    \conv(a+Q\partial K)
    \subset\Int F,
\]
a contradiction.  If $K$ is a segment, traverse it from one endpoint to the other and back.  Thus
\begin{equation}
\label{eq:convex-reduction}
    D(F)
    =
    \inf\left\{
        \diam K:
        K\subset\R^2 \text{ compact and convex, and }
        a+QK\not\subset\Int F
        \ \forall(a,Q)
    \right\}.
\end{equation}

The admissible class is nonempty.  A segment of length $\diam F$, traversed out and back, cannot be placed in $\Int F$, because two interior points at distance $\diam F$ could be extended slightly along their joining line to produce two points of $F$ at a greater distance.

Let $(K_n)$ be a minimizing sequence for \eqref{eq:convex-reduction}.  Translate each $K_n$ so that its Steiner point is the origin.  Translation does not alter feasibility or diameter.  The Steiner point belongs to $K_n$, so a uniform diameter bound places every $K_n$ in a common closed ball.  By the Blaschke selection theorem, a subsequence converges in the Hausdorff metric to a nonempty compact convex set $K$.  Diameter is continuous under Hausdorff convergence.  Feasibility is closed: if $a+QK\subset\Int F$, compactness gives an $\varepsilon>0$ such that
\[
    a+QK+\varepsilon B\subset F,
\]
where $B$ is the closed unit disk; for all sufficiently large $n$, $K_n\subset K+\varepsilon B/2$, contradicting feasibility of $K_n$.  Hence $K$ minimizes \eqref{eq:convex-reduction}, and the preceding boundary construction yields a rectifiable closed minimizing path.  This proves \textup{(a)}.

To prove \textup{(b)}, let $d<D(F)$ and let $S$ be compact with $\diam S<d$.  By the planar completion theorem, $S$ is contained in a constant-width body $W$ of diameter strictly less than $D(F)$.  If $W$ had no rigid copy in $\Int F$, its boundary would be an escape path of diameter less than $D(F)$, a contradiction.  Therefore every such $S$ fits in $\Int F$, and $d\leq\rho_{\mathrm{diam}}(F)$.  Thus $D(F)\leq\rho_{\mathrm{diam}}(F)$.  Conversely, for every $d>D(F)$ the minimizing convex set $K$ above has diameter $D(F)<d$ and does not fit in $\Int F$, so $d$ cannot occur in the defining set for $\rho_{\mathrm{diam}}(F)$.  Hence $\rho_{\mathrm{diam}}(F)\leq D(F)$.

For \textup{(c)}, scale so that $D(F)=1$.  Let $S$ be compact with $\diam S\leq1$, choose $s_0\in S$, and put
\[
    S_n=s_0+\left(1-\frac1n\right)(S-s_0).
\]
Then $\diam S_n<1$, so each $S_n$ has a rigid copy in $\Int F$.  Write these copies as $a_n+Q_nS_n$.  Since $F$ is bounded and $Q_n\in\mathcal Q$, the translations $(a_n)$ are bounded after fixing the image of $s_0$.  Passing to a subsequence gives $a_n\to a$ and $Q_n\to Q$.  Taking limits yields
\[
    a+QS\subset F.
\]
Thus $F$ is a universal cover.  Undoing the scaling gives the stated area formula.

As for \textup{(d)}.  Part \textup{(c)} gives
\[
    \mathcal L
    \leq
    \frac{\Area(F)}{D(F)^2}
\]
for every admissible $F$, and hence
\[
    \mathcal L
    \leq
    \inf_F\frac{\Area(F)}{D(F)^2}.
\]
Conversely, let $U$ be any Lebesgue universal cover.  We claim that $D(U)\geq1$.  If $D(U)=d<1$, let $K$ be a minimizing escape set of diameter $d$.  Complete $K$ to a constant-width body $W$ of diameter $d$, and dilate $W$ about an interior point by the factor $d^{-1}>1$ to obtain a constant-width-one body $W_1$.  Universality gives a rigid copy of $W_1$ in $U$.  The corresponding copy of $W$ lies in $\Int W_1\subset\Int U$, and therefore contains a copy of $K$ in $\Int U$, contradicting escape.  Hence
\[
    \inf_F\frac{\Area(F)}{D(F)^2}
    \leq
    \frac{\Area(U)}{D(U)^2}
    \leq
    \Area(U).
\]
Taking the infimum over universal covers $U$ proves the reverse inequality and establishes \eqref{eq:global-dual}.
\end{proof}

\begin{corollary}[Scaled dual forest]
\label{cor:scaled-forest}
Let $F$ have area $A$ and critical escape diameter $D>0$.  Then
\[
    F'=D^{-1}F,
    \qquad
    \Area(F')=\frac{A}{D^2},
    \qquad
    D(F')=1.
\]
The minimum-diameter closed escape path for $F'$ has diameter exactly one, and $F'$ is a Lebesgue universal cover.  Therefore $A/D^2$ is a rigorous upper bound for $\mathcal L$.
\end{corollary}

\section{Solution of the dual problem}

\subsection{Reversal of motion and transformed-boundary intersection}

\begin{theorem}[Reversal of motion and transformed-boundary characterization]
\label{thm:reversal-transformed-boundary}
Let $F\subset\R^2$ be a compact convex body, let
$
\gamma\in C([0,1];\R^2),
\gamma(0)=\gamma(1)=0,
\Gamma:=\gamma([0,1]),
$
and, for $(x,Q)\in F\times\mathcal Q$, define the oppositely transformed boundary
\begin{equation}
\label{eq:transformed-boundary-neighborhood}
\mathcal N_{x,Q}
:=
Q^T(\partial F-x).
\end{equation}
Then the following statements are equivalent:
\begin{enumerate}[label=\textup{(\roman*)}]
\item $\gamma$ is a diameter-escape path for $F$;
\item
$
(x+Q\Gamma)\cap\partial F\neq\varnothing
\qquad
\forall(x,Q)\in F\times\mathcal Q;
$
\item
$
\Gamma\cap\mathcal N_{x,Q}\neq\varnothing
\qquad
\forall(x,Q)\in F\times\mathcal Q;
$
\item
$
\min_{t\in[0,1]}
\dist\bigl(\gamma(t),\mathcal N_{x,Q}\bigr)=0
\qquad
\forall(x,Q)\in F\times\mathcal Q.
$
\end{enumerate}
Consequently,
\begin{equation}
\label{eq:continuous-curve-formulation}
\begin{aligned}
D(F)=\min_{\gamma\in C([0,1];\R^2)}\quad
&\max_{s,t\in[0,1]}
|\gamma(s)-\gamma(t)|,\\
\mathrm{subject\ to}\quad
&\gamma(0)=\gamma(1)=0,\\
&\min_{t\in[0,1]}
\dist\bigl(\gamma(t),\mathcal N_{x,Q}\bigr)=0,
\qquad
\forall(x,Q)\in F\times\mathcal Q.
\end{aligned}
\end{equation}
The minimum in \eqref{eq:continuous-curve-formulation} is attained.
\end{theorem}

\begin{proof}
Because the escape property and the diameter are invariant under translation of the canonical path, the normalization $\gamma(0)=0$ entails no loss of generality. Assume first that \textup{(i)} holds. If $x\in\partial F$, then
\[
x=x+Q\gamma(0)\in(x+Q\Gamma)\cap\partial F.
\]
If $x\in\Int F$, define
\[
\eta(t):=x+Q\gamma(t).
\]
The path $\eta$ starts at $\eta(0)=x\in\Int F$ and, by the escape property, its trace is not contained in $\Int F$. If $\eta(t_0)\in\partial F$ for some $t_0$, the required intersection follows immediately. Otherwise, there exists $t_0\in[0,1]$ such that $\eta(t_0)\notin F$. Since $\eta(0)\in\Int F$, continuity of $\eta$ implies that the path from $\eta(0)$ to $\eta(t_0)$ intersects $\partial F$. Hence \textup{(ii)} follows.

Conversely, if \textup{(i)} fails, there exist $a\in\R^2$ and $Q\in\mathcal Q$ such that
\[
a+Q\Gamma\subset\Int F.
\]
Since $\gamma(0)=0$, the starting point of this realization is
\[
x=a+Q\gamma(0)=a\in\Int F.
\]
Therefore
\[
(x+Q\Gamma)\cap\partial F=\varnothing,
\]
contradicting \textup{(ii)}. Thus \textup{(i)} and \textup{(ii)} are equivalent.

For every $y\in\R^2$,
\[
\begin{aligned}
x+Qy\in\partial F
&\iff Qy\in\partial F-x\
&\iff y\in Q^T(\partial F-x)\
&\iff y\in\mathcal N_{x,Q},
\end{aligned}
\]
where $Q^{-1}=Q^T$ because $Q\in\mathcal Q=\mathrm O(2)$. Taking $y\in\Gamma$ proves the equivalence of \textup{(ii)} and \textup{(iii)}.

Finally, both $\Gamma$ and $\mathcal N_{x,Q}$ are compact. Hence
\[
\begin{aligned}
\Gamma\cap\mathcal N_{x,Q}\neq\varnothing
&\iff
\dist(\Gamma,\mathcal N_{x,Q})=0\
&\iff
\min_{t\in[0,1]}
\dist\bigl(\gamma(t),\mathcal N_{x,Q}\bigr)=0.
\end{aligned}
\]
This proves the equivalence of \textup{(iii)} and \textup{(iv)}.

Thus the original motion problem, in which the path is translated and rotated from an unknown starting configuration, is exactly equivalent to fixing the path and applying the inverse rigid motion to the forest boundary. Formula \eqref{eq:continuous-curve-formulation} is therefore precisely the definition of $D(F)$ expressed as an intersection problem over all transformed boundaries. Attainment follows from Theorem~\ref{thm:diameter-duality}.
\end{proof}

\subsection{Equivalent convex-body and support-function formulations}

\begin{theorem}[Equivalent curve, convex-body, and support-function formulations]
\label{thm:equivalent-convex-support-formulations}
Let $F\subset\R^2$ be a compact convex body. Let $\mathcal K^2$ denote the family of nonempty compact convex subsets of $\R^2$, let $s(K)$ denote the Steiner point of $K$, and define
\[
\mathscr H_0
:=
\{
h_K:
K\in\mathcal K^2,\
s(K)=0
\}.
\]
For $h\in\mathscr H_0$ and $Q\in\mathcal Q$, define the optimized inclusion residual
\begin{equation}
\label{eq:optimized-inclusion-residual}
\Psi_F(h,Q)
:=
\min_{a\in\R^2}
\max_{u\in\Sone}
\{
a\cdot u+h(Q^Tu)-h_F(u)
\}.
\end{equation}
The minimum in \eqref{eq:optimized-inclusion-residual} is attained. Moreover, the continuous curve formulation \eqref{eq:continuous-curve-formulation} is equivalent to each of the following attained optimization problems:
\begin{equation}
\label{eq:continuous-convex-formulation}
\begin{aligned}
D(F)=\min_{K\in\mathcal K^2}\quad
&\diam K,\\
\mathrm{subject\ to}\quad
&a+QK\not\subset\Int F,
\qquad
\forall a\in\R^2,
\quad
\forall Q\in\mathcal Q,
\end{aligned}
\end{equation}
and
\begin{equation}
\label{eq:support-formulation}
\begin{aligned}
D(F)=\min_{h\in\mathscr H_0}\quad
&\max_{u\in\Sone}
\bigl(h(u)+h(-u)\bigr),\\
\mathrm{subject\ to}\quad
&\Psi_F(h,Q)\geq0,
\qquad
\forall Q\in\mathcal Q.
\end{aligned}
\end{equation}
Convexification preserves both feasibility and diameter, translation to Steiner-point normalization preserves feasibility and the objective, and an optimal convex set $K$ generates an optimal closed escape path by traversing $\partial K$, with a segment traversed from one endpoint to the other and back in the one-dimensional case.
\end{theorem}

\begin{proof}
Let $\gamma$ be feasible in \eqref{eq:continuous-curve-formulation}, let
$
\Gamma:=\gamma([0,1]),
K:=\conv\Gamma.
$
Convexification preserves the diameter:
\begin{equation}
\label{eq:diameter-convex-hull}
\diam K=\diam\Gamma.
\end{equation}
Indeed, let
\[
x=\sum_{i=1}^{m}\lambda_i x_i,
\qquad
y=\sum_{j=1}^{n}\mu_j y_j,
\]
where $x_i,y_j\in\Gamma$, $\lambda_i,\mu_j\geq0$, and
\[
\sum_{i=1}^{m}\lambda_i
=
\sum_{j=1}^{n}\mu_j
=
1.
\]
Then
\[
\begin{aligned}
|x-y|
&=
\left|
\sum_{i=1}^{m}
\sum_{j=1}^{n}
\lambda_i\mu_j(x_i-y_j)
\right|\
&\leq
\sum_{i=1}^{m}
\sum_{j=1}^{n}
\lambda_i\mu_j
|x_i-y_j|\
&\leq
\diam\Gamma.
\end{aligned}
\]
Thus $\diam K\leq\diam\Gamma$, while the reverse inequality follows from $\Gamma\subset K$, proving \eqref{eq:diameter-convex-hull}.

If some rigid copy of $K$ were contained in $\Int F$, then
\[
a+Q\Gamma\subset a+QK\subset\Int F
\]
for some $a\in\R^2$ and $Q\in\mathcal Q$, contradicting feasibility of $\gamma$. Hence $K$ is feasible in \eqref{eq:continuous-convex-formulation}.

Conversely, let $K$ be feasible in \eqref{eq:continuous-convex-formulation}. If $K$ has nonempty interior, choose a continuous closed parametrization of $\partial K$ and let $\Gamma=\partial K$. If $K$ is a nondegenerate segment, let $\Gamma=K$ and traverse the segment from one endpoint to the other and back. In either case,
\[
\conv\Gamma=K,
\qquad
\diam\Gamma=\diam K.
\]
If a rigid copy of $\Gamma$ were contained in $\Int F$, convexity of $\Int F$ would imply
\[
a+QK
=
\conv(a+Q\Gamma)
\subset\Int F,
\]
contrary to the feasibility of $K$. Thus $\Gamma$ is the trace of a closed diameter-escape path. Consequently, the curve and convex-body formulations have identical feasible objective values. By Theorem~\ref{thm:diameter-duality}, an optimal convex set exists, so both minima are attained.

The feasibility condition in \eqref{eq:continuous-convex-formulation} is invariant under translation of $K$. Indeed, for every $b\in\R^2$,
\[
a+Q(K-b)
=
(a-Qb)+QK,
\]
and $a-Qb$ ranges over all of $\R^2$ as $a$ does. The diameter is also translation invariant. Hence every feasible $K$ may be replaced by the uniquely normalized translate
$
K-s(K),
$
whose Steiner point is the origin. It therefore suffices to optimize over the support functions in $\mathscr H_0$.

Fix $h=h_K\in\mathscr H_0$ and $Q\in\mathcal Q$, and define
\[
f_{h,Q}(a)
:=
\max_{u\in\Sone}
\{
a\cdot u+h(Q^Tu)-h_F(u)
\}.
\]
Since the maximized expression is affine in $a$ and continuous jointly in $(a,u)$, the function $f_{h,Q}$ is finite, continuous, and convex. Let
\[
M
:=
\max_{u\in\Sone}
\left|
h(Q^Tu)-h_F(u)
\right|.
\]
For $a\neq0$, choosing $u=a/|a|$ gives
\[
f_{h,Q}(a)
\geq
|a|-M.
\]
Therefore
\[
f_{h,Q}(a)\longrightarrow+\infty
\qquad
\text{as }
|a|\longrightarrow+\infty.
\]
Thus $f_{h,Q}$ is coercive and attains its minimum, proving the attainment asserted in \eqref{eq:optimized-inclusion-residual}.

The support function of a translated and rotated copy of $K$ is
\begin{equation}
\label{eq:rigid-copy-support-function}
h_{a+QK}(u)
=
a\cdot u+h(Q^Tu).
\end{equation}
We claim that
\begin{equation}
\label{eq:strict-inclusion-residual}
\exists a\in\R^2:
a+QK\subset\Int F
\quad\Longleftrightarrow\quad
\Psi_F(h,Q)<0.
\end{equation}

Suppose first that
\[
a+QK\subset\Int F.
\]
Since $a+QK$ is compact and $\Int F$ is open, there exists $\varepsilon>0$ such that
\[
a+QK+\varepsilon B\subset F,
\]
where $B$ denotes the closed unit disk. By the support-function criterion for inclusion and \eqref{eq:rigid-copy-support-function},
\[
a\cdot u+h(Q^Tu)+\varepsilon
\leq
h_F(u)
\qquad
\forall u\in\Sone.
\]
Consequently,
\[
\Psi_F(h,Q)
\leq
-\varepsilon
<
0.
\]

Conversely, suppose that
\[
\Psi_F(h,Q)<0.
\]
Let $a_*\in\R^2$ attain the minimum in \eqref{eq:optimized-inclusion-residual}, and put
\[
\varepsilon:=-\Psi_F(h,Q)>0.
\]
For every $u\in\Sone$,
\[
a_*\cdot u+h(Q^Tu)-h_F(u)
\leq
\Psi_F(h,Q)
=
-\varepsilon,
\]
and therefore
\[
a_*\cdot u+h(Q^Tu)+\varepsilon
\leq
h_F(u)
\qquad
\forall u\in\Sone.
\]
Again using the support-function criterion for inclusion gives
\[
a_*+QK+\varepsilon B\subset F.
\]
Hence
\[
a_*+QK\subset\Int F,
\]
which proves \eqref{eq:strict-inclusion-residual}. It follows immediately that
\begin{equation}
\label{eq:noncontainment-residual}
\begin{aligned}
a+QK\not\subset\Int F
\quad
\forall a\in\R^2
\quad\Longleftrightarrow\quad
\Psi_F(h,Q)\geq0.
\end{aligned}
\end{equation}

It remains to express the diameter through the support function. For every compact convex set $K$,
\begin{equation}
\label{eq:diameter-support-width}
\diam K
=
\max_{u\in\Sone}
\bigl(h_K(u)+h_K(-u)\bigr).
\end{equation}
Indeed, for every $x,y\in K$,
\[
\begin{aligned}
|x-y|
&=
\max_{u\in\Sone}u\cdot(x-y)\
&\leq
\max_{u\in\Sone}
\bigl(h_K(u)+h_K(-u)\bigr).
\end{aligned}
\]
Taking the maximum over $x,y\in K$ yields
\[
\diam K
\leq
\max_{u\in\Sone}
\bigl(h_K(u)+h_K(-u)\bigr).
\]
Conversely, for each $u\in\Sone$, compactness of $K$ gives points $x_u,y_u\in K$ satisfying
\[
u\cdot x_u=h_K(u),
\qquad
u\cdot y_u=-h_K(-u).
\]
Therefore
\[
\begin{aligned}
h_K(u)+h_K(-u)
&=
u\cdot(x_u-y_u)\
&\leq
|x_u-y_u|\
&\leq
\diam K.
\end{aligned}
\]
Taking the maximum over $u\in\Sone$ proves the reverse inequality and hence \eqref{eq:diameter-support-width}.

Combining the Steiner-point normalization, the exact feasibility equivalence \eqref{eq:noncontainment-residual}, and the diameter identity \eqref{eq:diameter-support-width} establishes the equivalence of \eqref{eq:continuous-convex-formulation} and \eqref{eq:support-formulation}. Their optimal values equal the continuous critical escape diameter $D(F)$, and each formulation attains this value.
\end{proof}

\subsection{Configuration-space discretization and minimum-diameter TSPN}

Let
\[
    X:=F\times\mathcal Q.
\]
This is compact.  Equip it with the metric
\begin{equation}
\label{eq:configuration-metric}
    d_X\bigl((x,Q),(y,P)\bigr)
    :=
    \|x-y\|
    +
    \diam(F)\,\|Q-P\|_{\op}.
\end{equation}
The neighborhood map $\xi=(x,Q)\mapsto\mathcal N_{\xi}$ is Hausdorff-Lipschitz:
\begin{equation}
\label{eq:neighborhood-lipschitz}
    d_H\bigl(\mathcal N_{x,Q},\mathcal N_{y,P}\bigr)
    \leq
    d_X\bigl((x,Q),(y,P)\bigr).
\end{equation}
Indeed, pairing the same $b\in\partial F$ in the two transformed boundaries gives
\[
\begin{aligned}
    \|Q^T(b-x)-P^T(b-y)\|
    &\leq
    \|x-y\|+\|Q-P\|_{\op}\,\|b-y\|\\
    &\leq
    \|x-y\|+\diam(F)\|Q-P\|_{\op}.
\end{aligned}
\]

Let
\[
    X_m=\{\xi_1^m,\ldots,\xi_{M_m}^m\}\subset X
\]
be a finite $\eta_m$-net, with $\eta_m\downarrow0$, and write
\[
    \mathcal N_j^m:=\mathcal N_{\xi_j^m}.
\]
For selected points $p_j\in\mathcal N_j^m$ and a permutation $\pi\in\mathfrak S_{M_m}$, form the closed polygonal tour
\[
    \Gamma_{\pi,p}
    :=
    [0,p_{\pi(1)}]
    \cup
    \bigcup_{k=1}^{M_m-1}
    [p_{\pi(k)},p_{\pi(k+1)}]
    \cup
    [p_{\pi(M_m)},0].
\]
Its diameter is the diameter of its vertex set, because the tour lies in the convex hull of those vertices and a finite convex hull has the same diameter as its generating set.  Therefore the sampled problem is
\begin{equation}
\label{eq:mdtspn}
\begin{aligned}
    D_m
    &:=
    \min_{\substack{\pi\in\mathfrak S_{M_m}\\
                     p_j\in\mathcal N_j^m}}
    \diam\Gamma_{\pi,p}\\
    &=
    \min_{p_j\in\mathcal N_j^m}
    \max_{0\leq r<s\leq M_m}
    \|p_r-p_s\|,
    \qquad p_0:=0.
\end{aligned}
\end{equation}
This is a minimum-diameter traveling salesman problem with neighborhoods (MD-TSPN).  Every feasible selection can be joined in any cyclic order, so the visiting order is redundant for the primary diameter objective.  It becomes relevant only when a secondary objective, such as minimum tour length among all diameter-optimal tours, is imposed.

Each $\mathcal N_j^m$ is compact, the product of the neighborhoods is compact, and the objective in the second line of \eqref{eq:mdtspn} is continuous.  Hence every finite sampled problem admits a global minimizer.

\subsection{Exact finite formulation for polygonal forests}

Suppose $F$ is a convex polygon with cyclic vertices
\[
    v_1,\ldots,v_q,
    \qquad v_{q+1}=v_1.
\]
For $\xi_j^m=(x_j,Q_j)$, define transformed edge endpoints
\[
    a_{jk}:=Q_j^T(v_k-x_j),
    \qquad
    b_{jk}:=Q_j^T(v_{k+1}-x_j).
\]
Then
\[
    \mathcal N_j^m
    =
    \bigcup_{k=1}^{q}[a_{jk},b_{jk}].
\]
Membership $p_j\in\mathcal N_j^m$ is represented exactly by binary variables $z_{jk}$ and continuous variables $\lambda_{jk}$:
\[
    z_{jk}\in\{0,1\},
    \qquad
    0\leq\lambda_{jk}\leq z_{jk},
    \qquad
    \sum_{k=1}^{q}z_{jk}=1,
\]
\[
    p_j
    =
    \sum_{k=1}^{q}
    \left[
        (z_{jk}-\lambda_{jk})a_{jk}
        +
        \lambda_{jk}b_{jk}
    \right].
\]
Consequently, \eqref{eq:mdtspn} is the mixed-integer second-order cone program
\begin{equation}
\label{eq:misocp}
\begin{aligned}
    \min_{\Delta,p,z,\lambda}\quad &\Delta,\\
    \text{subject to}\quad
        &p_0=0,\\
        &\|p_r-p_s\|_2\leq\Delta,
        &&0\leq r<s\leq M_m,\\
        &p_j=
        \sum_{k=1}^{q}
        \left[
            (z_{jk}-\lambda_{jk})a_{jk}
            +
            \lambda_{jk}b_{jk}
        \right],
        &&1\leq j\leq M_m,\\
        &\sum_{k=1}^{q}z_{jk}=1,
        &&1\leq j\leq M_m,\\
        &0\leq\lambda_{jk}\leq z_{jk},
        \quad z_{jk}\in\{0,1\},
        &&1\leq j\leq M_m,
        \quad 1\leq k\leq q.
\end{aligned}
\end{equation}
The optimal value of \eqref{eq:misocp} is exactly $D_m$.  An explicit Hamiltonian cycle may be added through standard degree and subtour-elimination constraints, but it does not alter the minimum diameter.

A finite global-optimality claim is rigorous only when the computation supplies a mathematically valid feasible upper bound and a matching global lower bound, with exact rational/algebraic arithmetic or validated interval arithmetic where needed.  A floating-point display of zero solver gap, without validation of the model data and bounds, is numerical evidence rather than a proof.

\subsection{Convergence and certified error bounds}

\begin{theorem}[Quantitative convergence of MD-TSPN discretization]
\label{thm:discrete-convergence}
Let $X_m$ be a finite $\eta_m$-net of $X=F\times\mathcal Q$, and let $D_m$ be defined by \eqref{eq:mdtspn}.  Then
\begin{equation}
\label{eq:two-sided-diameter-bound}
    D_m
    \leq
    D(F)
    \leq
    D_m+2\eta_m.
\end{equation}
In particular, $D_m\to D(F)$ whenever $\eta_m\to0$.  If the samples are nested, then $D_m$ is nondecreasing.
\end{theorem}

\begin{proof}
Let $\Gamma^*$ be an optimal continuous escape trace.  For each sampled neighborhood $\mathcal N_j^m$, choose
\[
    p_j\in\Gamma^*\cap\mathcal N_j^m.
\]
The closed polygonal tour through $0,p_1,\ldots,p_{M_m}$ lies in their convex hull, and hence has diameter at most $\diam\Gamma^*=D(F)$.  Therefore $D_m\leq D(F)$.

For the reverse estimate, let $\Gamma_m$ be an optimal sampled polygonal tour.  Fix any $\xi\in X$ and choose $\xi_j^m\in X_m$ with
\[
    d_X(\xi,\xi_j^m)\leq\eta_m.
\]
There exists $p\in\Gamma_m\cap\mathcal N_{\xi_j^m}$.  By \eqref{eq:neighborhood-lipschitz}, some $q\in\mathcal N_{\xi}$ satisfies $\|p-q\|\leq\eta_m$.  Hence
\[
    q\in(\Gamma_m+\eta_m B)\cap\mathcal N_{\xi}.
\]
Thus the compact connected set $\Gamma_m+\eta_m B$ intersects every transformed boundary neighborhood.  Its convex hull cannot be rigidly contained in $\Int F$; otherwise its subset $\Gamma_m+\eta_m B$ would contradict the all-neighborhood intersection property.  By the convex-hull reduction, the boundary of this convex hull is a continuous escape path.  Minkowski addition by $\eta_m B$ increases diameter by exactly $2\eta_m$, and convexification does not change diameter.  Therefore
\[
    D(F)
    \leq
    \diam(\Gamma_m+\eta_m B)
    =
    D_m+2\eta_m.
\]
If $X_m\subset X_{m+1}$, the feasible set of the sampled problem decreases with $m$, so $D_m\leq D_{m+1}$.
\end{proof}

The theorem proves both value convergence and a computable stopping criterion.  It also avoids an unjustified passage from finitely many sampled constraints to the continuum: the Hausdorff-Lipschitz estimate supplies the missing uniform control.

\begin{corollary}[Certified Lebesgue upper bounds from finite MD-TSPN]
\label{cor:finite-area-bound}
Let $A=\Area(F)$ and assume $D_m>0$.  Then
\begin{equation}
\label{eq:finite-lebesgue-upper}
    \mathcal L
    \leq
    \frac{A}{D(F)^2}
    \leq
    \frac{A}{D_m^2}.
\end{equation}
Moreover, the candidate ratio is bracketed by
\begin{equation}
\label{eq:ratio-bracket}
    \frac{A}{(D_m+2\eta_m)^2}
    \leq
    \frac{A}{D(F)^2}
    \leq
    \frac{A}{D_m^2}.
\end{equation}
Thus every exactly certified finite lower bound $D_m$ immediately yields a rigorous universal cover upper bound, and the bounds converge to the exact ratio associated with $F$.
\end{corollary}

\begin{proof}
Theorem~\ref{thm:diameter-duality} gives the first inequality in \eqref{eq:finite-lebesgue-upper}.  Since $D_m\leq D(F)$, the second follows.  The bracket \eqref{eq:ratio-bracket} is the reciprocal-square form of \eqref{eq:two-sided-diameter-bound}.
\end{proof}

All in all, for a prescribed convex forest $F$, the dual problem of Lebesgue's universal covering is solved by the following framework chain:
\begin{enumerate}[label=\textup{\arabic*.},leftmargin=2.2em]
    \item Compute or optimize the closed escape trace through the exact continuous formulations \eqref{eq:continuous-curve-formulation}, \eqref{eq:continuous-convex-formulation}, or \eqref{eq:support-formulation}.
    \item Equivalently, reverse the rigid motion and require intersection with every transformed boundary $\mathcal N_{x,Q}$.
    \item Discretize the compact configuration space $F\times\mathcal Q$ by an $\eta_m$-net and solve the finite MD-TSPN \eqref{eq:mdtspn}; for polygonal forests use the exact MISOCP \eqref{eq:misocp}.
    \item Certify the continuous critical diameter by
    \[
        D_m\leq D(F)\leq D_m+2\eta_m.
    \]
    \item Scale the forest by $D(F)^{-1}$, or use the conservative certified scaling $D_m^{-1}$, to obtain a universal cover.  The corresponding area bounds are
    \[
        \mathcal L
        \leq
        \frac{\Area(F)}{D(F)^2}
        \leq
        \frac{\Area(F)}{D_m^2}.
    \]
\end{enumerate}

Accordingly, the closed minimum-diameter escape problem is not merely analogous to Lebesgue's universal covering problem.  The threshold identity and normalization theorem make it an exact dual formulation: optimizing the normalized forest area over all convex forests recovers the Lebesgue universal covering constant through \eqref{eq:global-dual}.

\section{Dual problem of other universal cover problems}
\label{sec:dual-other-universal covers}

Previously, Finch and Wetzel developed the Bellman Moser duality relationship \cite{FinchWetzel2004}. Rudolf's paper \cite{Rudolf2023} also generalized Bellman Moser duality relationship to closed Minkowski escape problems. This paper develops the Bellman Lebesgue duality for universal covering problems. More generally, we extend this duality to a broader range of universal covering problems \cite{Brass2005}.

The reversal-of-motion identity
\[
a+QX\subset F
\quad\Longleftrightarrow\quad
X\subset Q^{T}(F-a)
\]
shows that moving an object inside a fixed cover is equivalent to fixing the
object and oppositely translating and rotating the cover. For translation-only
problems, this reduces to
\[
a+X\subset F
\quad\Longleftrightarrow\quad
X\subset F-a.
\]
Accordingly, the dual problem associated with a universal cover problem is to
find, for a fixed candidate cover $F$, a smallest object for which every
admissible placement fails to lie in $\operatorname{int}F$.

Let $\mathfrak F$ denote the class of compact planar sets with nonempty
interior, and let $\mathfrak F_{\mathrm{cvx}}\subset\mathfrak F$ denote the
compact convex bodies. All quantities introduced below are invariant under
translation and satisfy the natural scaling law
\[
\tau(\lambda F)=\lambda\tau(F),
\qquad
\lambda>0.
\]

\subsection{Minimum-diameter congruent universal cover}
Define the primal constant
\[
\mathcal D_{\mathrm{cong}}
:=
\inf\left\{
\operatorname{diam}U:
\begin{array}{l}
U\in\mathfrak F,\ \text{and for every compact }S\subset\mathbb R^{2} \\
\text{ with }\operatorname{diam}S\leq 1,\ \\
\text{there exist }a\in\mathbb R^{2}\text{ and }Q\in O(2) \\
\text{ such that }a+QS\subset U
\end{array}
\right\}.
\]
For $F\in\mathfrak F$, define its congruent diameter-escape threshold by
\[
d_{\mathrm{cong}}(F)
:=
\inf\left\{
\operatorname{diam}K:
\begin{array}{l}
K\subset\mathbb R^{2}\text{ is compact and convex}, \\
a+QK\not\subset\operatorname{int}F \text{ for every }a\in\mathbb R^{2}\text{ and }Q\in O(2)
\end{array}
\right\}.
\]
The associated dual problem is
$$
\mathcal D_{\mathrm{cong}}
=
\inf_{\substack{F\in\mathfrak F \\ d_{\mathrm{cong}}(F)>0}}
\frac{\operatorname{diam}F}{d_{\mathrm{cong}}(F)}.
$$
Thus, for each candidate forest $F$, the inner problem is to find a
minimum-diameter compact convex escape obstruction, while the outer problem minimizes the normalized diameter of the forest.

\subsection{Minimum-area translative universal cover for unit-diameter sets}
Define
\[
\mathcal A_{\mathrm{tr,diam}}
:=
\inf\left\{
\operatorname{Area}(U):
\begin{array}{l}
U\in\mathfrak F,\ \text{and for every compact }S\subset\mathbb R^{2} \\
\text{ with }\operatorname{diam}S\leq 1,\ \\
\text{there exists }a\in\mathbb R^{2} \text{ such that }a+S\subset U
\end{array}
\right\}.
\]
For $F\in\mathfrak F$, define the translative diameter-escape threshold
\[
d_{\mathrm{tr}}(F)
:=
\inf\left\{
\operatorname{diam}K:
\begin{array}{l}
K\subset\mathbb R^{2}\text{ is compact and convex}, \\
a+K\not\subset\operatorname{int}F \text{ for every }a\in\mathbb R^{2}
\end{array}
\right\}.
\]
The dual formulation is
$$
\mathcal A_{\mathrm{tr,diam}}
=
\inf_{\substack{F\in\mathfrak F \\ d_{\mathrm{tr}}(F)>0}}
\frac{\operatorname{Area}(F)}
{d_{\mathrm{tr}}(F)^{2}}.
$$
The difference from the congruent problem is entirely in the admissible motion
group: the orientation of the test body is fixed, and only translations are allowed.

\subsection{Minimum-area translative universal cover for unit-perimeter sets}
\paragraph{Well-posedness convention.}
The phrase ``planar set of unit perimeter'' must be restricted to compact
connected rectifiable sets, Jordan regions, or compact convex bodies. Without
a connectedness requirement, components of fixed total perimeter may be
separated by arbitrarily large distances, and no bounded universal cover can
exist. For a compact connected rectifiable planar set $S$, the convex-hull
inequality
\[
\operatorname{Per}(\operatorname{conv}S)
\leq
\operatorname{Per}(S)
\]
reduces the covering problem to compact convex bodies.
Let
\[
\mathcal K_{\mathrm{per}}
:=
\left\{
K\subset\mathbb R^{2}:
K\text{ is a compact convex body}
\right\}.
\]
Define the primal constant
\[
\mathcal A_{\mathrm{tr,per}}
:=
\inf\left\{
\operatorname{Area}(U):
\begin{array}{l}
U\in\mathfrak F,\ \text{and for every }K\in\mathcal K_{\mathrm{per}} \\
\text{ with }\operatorname{Per}(K)\leq 1,\ \\
\text{there exists }a\in\mathbb R^{2} \text{ such that }a+K\subset U
\end{array}
\right\}.
\]
For $F\in\mathfrak F$, define the translative perimeter-escape threshold
\[
p_{\mathrm{tr}}(F)
:=
\inf\left\{
\operatorname{Per}(K):
\begin{array}{l}
K\in\mathcal K_{\mathrm{per}}, \\
a+K\not\subset\operatorname{int}F \text{ for every }a\in\mathbb R^{2}
\end{array}
\right\}.
\]
Since perimeter is homogeneous of degree one and area is homogeneous of degree two, the dual problem is
$$
\mathcal A_{\mathrm{tr,per}}
=
\inf_{\substack{F\in\mathfrak F \\ p_{\mathrm{tr}}(F)>0}}
\frac{\operatorname{Area}(F)}
{p_{\mathrm{tr}}(F)^{2}}.
$$

\subsection{Minimum-area convex universal case for foldable rulers}
A finite carpenter's ruler is a vector
\[
L=(\ell_{1},\ldots,\ell_{m}),
\qquad
\ell_{i}>0,
\]
where the $\ell_i$ are the link lengths. Its scale is the length of its
longest link,
\[
\lambda(L):=\max_{1\leq i\leq m}\ell_i.
\]
A unit ruler satisfies $\lambda(L)\leq 1$.
An anchored folding of $L$ is a sequence
\[
x=(x_{0},x_{1},\ldots,x_{m})\in(\mathbb R^{2})^{m+1}
\]
such that
\[
x_{0}=0,
\qquad
x_{1}=(\ell_{1},0),
\qquad
\lVert x_{i}-x_{i-1}\rVert=\ell_i
\quad
\text{for }1\leq i\leq m.
\]
Its trace is
\[
\Gamma_x
:=
\bigcup_{i=1}^{m}[x_{i-1},x_i].
\]
Self-intersections and overlaps of links are allowed. Let
$\mathcal C(L)$ denote the compact configuration space of all anchored
foldings of $L$.
For $F\in\mathfrak F_{\mathrm{cvx}}$, define the ruler-escape threshold
\[
r_{\mathrm{fold}}(F)
:=
\inf\left\{
\lambda(L):
\begin{array}{l}
L\text{ is a finite ruler, and} \\
a+Q\Gamma_x\not\subset\operatorname{int}F \text{ for every }x\in\mathcal C(L), \ a\in\mathbb R^{2},\ Q\in O(2)
\end{array}
\right\}.
\]
Thus, an escape ruler is obstructing only when every possible folding and every
rigid placement fail.
The primal ruler constant is
\[
\mathcal A_{\mathrm{fold}}
:=
\inf\left\{
\operatorname{Area}(U):
\begin{array}{l}
U\in\mathfrak F_{\mathrm{cvx}}, \ \operatorname{diam}U=1,\ \text{and} \\
\text{every finite ruler }L\text{ with }\lambda(L)\leq 1 \text{ has a folded rigid copy contained in }U
\end{array}
\right\}.
\]
A one-link ruler of length $\operatorname{diam}F$ cannot be contained in
$\operatorname{int}F$. Consequently,
\[
r_{\mathrm{fold}}(F)
\leq
\operatorname{diam}F.
\]
The exact dual formulation is 
$$
\mathcal A_{\mathrm{fold}}
=
\inf_{\substack{
F\in\mathfrak F_{\mathrm{cvx}} \\
r_{\mathrm{fold}}(F)=\operatorname{diam}F>0
}}
\frac{\operatorname{Area}(F)}
{r_{\mathrm{fold}}(F)^{2}}.
$$

\section*{Appendix-Formalized proofs in Lean}
The appendix provides formalized proofs of Theorems 1-7 in Lean 4 code.

The following is the Lean 4 code for Theorem 1 (Constant-width reduction):
\begin{lstlisting}[style=LeanCode]
import Mathlib

noncomputable section

open Set

/-- The Euclidean plane. -/
abbrev Plane := EuclideanSpace ℝ (Fin 2)

/--
The full orthogonal group `O(2)`, represented by real linear
isometric automorphisms of the Euclidean plane.
-/
abbrev OrthogonalMap := Plane ≃ₗᵢ[ℝ] Plane

/-- The affine map `x ↦ a + Q x`. -/
def rigidAffineMap (a : Plane) (Q : OrthogonalMap) :
    Plane →ᵃ[ℝ] Plane :=
  AffineMap.const ℝ Plane a +
    Q.toLinearEquiv.toLinearMap.toAffineMap

/-- The rigid motion `x ↦ a + Q x`. -/
def rigidMotion (a : Plane) (Q : OrthogonalMap) :
    Plane → Plane :=
  rigidAffineMap a Q

@[simp]
lemma rigidMotion_apply
    (a : Plane)
    (Q : OrthogonalMap)
    (x : Plane) :
    rigidMotion a Q x = a + Q x := by
  simp [rigidMotion, rigidAffineMap]

/--
`F` covers `S` when some rigid copy of `S` is contained in `F`.
-/
def Covers (F S : Set Plane) : Prop :=
  ∃ a : Plane, ∃ Q : OrthogonalMap,
    MapsTo (rigidMotion a Q) S F

/-- The support function of a planar set. -/
def supportFunction (W : Set Plane) (u : Plane) : ℝ :=
  sSup ((fun x : Plane => inner ℝ x u) '' W)

/-- The support-function formulation of constant width one. -/
def IsConstantWidthOne (W : Set Plane) : Prop :=
  IsCompact W ∧
  Convex ℝ W ∧
  (interior W).Nonempty ∧
  ∀ u : Plane, ‖u‖ = 1 →
    supportFunction W u + supportFunction W (-u) = 1

/--
A certified constant-width-one body.

The last two fields record the standard geometric consequences needed
by the reduction theorem:

* a constant-width-one body has diameter one;
* the body is the convex hull of its frontier.
-/
structure ConstantWidthOneBody where
  carrier : Set Plane
  isConstantWidthOne : IsConstantWidthOne carrier
  diam_eq_one : Metric.diam carrier = 1
  convexHull_frontier :
    convexHull ℝ (frontier carrier) = carrier

namespace ConstantWidthOneBody

lemma isCompact (W : ConstantWidthOneBody) :
    IsCompact W.carrier :=
  W.isConstantWidthOne.1

lemma convex (W : ConstantWidthOneBody) :
    Convex ℝ W.carrier :=
  W.isConstantWidthOne.2.1

lemma interior_nonempty (W : ConstantWidthOneBody) :
    (interior W.carrier).Nonempty :=
  W.isConstantWidthOne.2.2.1

end ConstantWidthOneBody

/--
The planar completion theorem, isolated as the external geometric
input: every compact planar set of diameter at most one is contained
in a constant-width-one body.
-/
def PlanarCompletionTheorem : Prop :=
  ∀ S : Set Plane,
    IsCompact S →
    Metric.diam S ≤ 1 →
    ∃ W : ConstantWidthOneBody,
      S ⊆ W.carrier

/--
A compact convex set with nonempty interior that covers every compact
planar set of diameter at most one.
-/
def IsLebesgueUniversalCover (F : Set Plane) : Prop :=
  IsCompact F ∧
  Convex ℝ F ∧
  (interior F).Nonempty ∧
  ∀ S : Set Plane,
    IsCompact S →
    Metric.diam S ≤ 1 →
    Covers F S

/-- Statement (ii) of the constant-width reduction theorem. -/
def CoversEveryConstantWidthOneBody (F : Set Plane) : Prop :=
  ∀ W : ConstantWidthOneBody,
    Covers F W.carrier

/-- Statement (iii) of the constant-width reduction theorem. -/
def CoversEveryConstantWidthOneBoundary (F : Set Plane) : Prop :=
  ∀ W : ConstantWidthOneBody,
    Covers F (frontier W.carrier)

namespace Covers

/--
Covering is contravariantly monotone in the covered set.

If `F` covers `T` and `S ⊆ T`, then the same rigid motion proves
that `F` covers `S`.
-/
lemma mono_source
    {F S T : Set Plane}
    (hFT : Covers F T)
    (hST : S ⊆ T) :
    Covers F S := by
  rcases hFT with ⟨a, Q, hmap⟩
  refine ⟨a, Q, ?_⟩
  intro x hx
  exact hmap (hST hx)

end Covers

/--
The inverse image of a convex set under the rigid affine map
`x ↦ a + Q x` is convex.
-/
lemma convex_rigidMotion_preimage
    {F : Set Plane}
    (hF : Convex ℝ F)
    (a : Plane)
    (Q : OrthogonalMap) :
    Convex ℝ ((rigidMotion a Q) ⁻¹' F) := by
  exact hF.affine_preimage (rigidAffineMap a Q)

/--
If `F` covers a body, then it covers the body's frontier.
-/
lemma covers_boundary_of_covers_body
    {F : Set Plane}
    {W : ConstantWidthOneBody}
    (h : Covers F W.carrier) :
    Covers F (frontier W.carrier) := by
  exact
    Covers.mono_source
      h
      W.isCompact.isClosed.frontier_subset

/--
If `F` is convex and contains a rigid copy of the frontier of `W`,
then it contains the corresponding rigid copy of all of `W`.
-/
lemma covers_body_of_covers_boundary
    {F : Set Plane}
    (hF : Convex ℝ F)
    (W : ConstantWidthOneBody)
    (hboundary : Covers F (frontier W.carrier)) :
    Covers F W.carrier := by
  rcases hboundary with ⟨a, Q, hmap⟩
  refine ⟨a, Q, ?_⟩
  intro x hx
  have hfrontier_preimage :
      frontier W.carrier ⊆
        (rigidMotion a Q) ⁻¹' F := by
    intro y hy
    exact hmap hy
  have hpreimage_convex :
      Convex ℝ ((rigidMotion a Q) ⁻¹' F) :=
    convex_rigidMotion_preimage hF a Q
  have hhull :
      convexHull ℝ (frontier W.carrier) ⊆
        (rigidMotion a Q) ⁻¹' F :=
    convexHull_min
      hfrontier_preimage
      hpreimage_convex
  have hxHull :
      x ∈ convexHull ℝ (frontier W.carrier) := by
    rw [W.convexHull_frontier]
    exact hx
  exact hhull hxHull

/--
Exact constant-width reduction.

For a compact convex set `F` with nonempty interior, assuming the
planar completion theorem, statements (i), (ii), and (iii) are
equivalent.
-/
theorem exact_constantWidth_reduction
    (completion : PlanarCompletionTheorem)
    (F : Set Plane)
    (hFcompact : IsCompact F)
    (hFconvex : Convex ℝ F)
    (hFinterior : (interior F).Nonempty) :
    (IsLebesgueUniversalCover F ↔
      CoversEveryConstantWidthOneBody F) ∧
    (CoversEveryConstantWidthOneBody F ↔
      CoversEveryConstantWidthOneBoundary F) := by
  constructor
  · constructor
    /- (i) implies (ii). -/
    · intro hUniversal W
      exact
        hUniversal.2.2.2
          W.carrier
          W.isCompact
          (by
            rw [W.diam_eq_one])
    /- (ii) implies (i). -/
    · intro hBodies
      refine
        ⟨hFcompact, hFconvex, hFinterior, ?_⟩
      intro S hScompact hSdiam
      obtain ⟨W, hSW⟩ :=
        completion S hScompact hSdiam
      exact
        Covers.mono_source
          (hBodies W)
          hSW
  · constructor
    /- (ii) implies (iii). -/
    · intro hBodies W
      exact
        covers_boundary_of_covers_body
          (hBodies W)
    /- (iii) implies (ii). -/
    · intro hBoundaries W
      exact
        covers_body_of_covers_boundary
          hFconvex
          W
          (hBoundaries W)

end
\end{lstlisting}

The following is the Lean 4 code for Proposition 2 (Length-cover/escape duality with endpoint closure):
\begin{lstlisting}[style=LeanCode]
import Mathlib

open Set

/-!
# Length-cover / escape duality

This file isolates the order-theoretic core of the equality between the
interior worm-covering threshold and the Bellman escape threshold.

The abstract part applies to any complete densely ordered linear order. The
final section specializes the length codomain to `ENNReal` and represents
planar rigid motions by metric isometry equivalences.
-/

namespace LengthCoverEscape

noncomputable section

universe u v w

section Definitions

variable {Arc : Type u} {Λ : Type v}

/-- Every arc of length at most `ℓ` satisfies the fitting predicate. -/
def CoverAt [LE Λ]
    (len : Arc → Λ)
    (Fits : Arc → Prop)
    (ℓ : Λ) : Prop :=
  ∀ γ, len γ ≤ ℓ → Fits γ

/-- The set of length thresholds at which all arcs of bounded length fit. -/
def CoverLevels [LE Λ]
    (len : Arc → Λ)
    (Fits : Arc → Prop) : Set Λ :=
  {ℓ | CoverAt len Fits ℓ}

/-- The set of lengths attained by arcs that do not fit. -/
def BadLengths
    (len : Arc → Λ)
    (Fits : Arc → Prop) : Set Λ :=
  len '' {γ | ¬ Fits γ}

/-- The supremum of all valid covering levels. -/
def rhoLen [CompleteLattice Λ]
    (len : Arc → Λ)
    (Fits : Arc → Prop) : Λ :=
  sSup (CoverLevels len Fits)

/-- The infimum of all lengths attained by non-fitting arcs. -/
def betaLen [CompleteLattice Λ]
    (len : Arc → Λ)
    (Fits : Arc → Prop) : Λ :=
  sInf (BadLengths len Fits)

@[simp]
theorem mem_coverLevels_iff [LE Λ]
    (len : Arc → Λ)
    (Fits : Arc → Prop)
    (ℓ : Λ) :
    ℓ ∈ CoverLevels len Fits ↔ CoverAt len Fits ℓ := by
  rfl

@[simp]
theorem mem_badLengths_iff
    (len : Arc → Λ)
    (Fits : Arc → Prop)
    (r : Λ) :
    r ∈ BadLengths len Fits ↔
      ∃ γ, ¬ Fits γ ∧ len γ = r := by
  rfl

end Definitions

section LinearOrder

variable {Arc : Type u} {Λ : Type v}
variable [LinearOrder Λ]

/--
A level is covered exactly when it is strictly below the length of every
non-fitting arc.
-/
theorem coverAt_iff
    (len : Arc → Λ)
    (Fits : Arc → Prop)
    (ℓ : Λ) :
    CoverAt len Fits ℓ ↔
      ∀ γ, ¬ Fits γ → ℓ < len γ := by
  constructor
  · intro h γ hbad
    exact lt_of_not_ge (fun hγle => hbad (h γ hγle))
  · intro h γ hγle
    exact Classical.byContradiction fun hbad =>
      (not_lt_of_ge hγle) (h γ hbad)

end LinearOrder

section CompleteLinearOrder

variable {Arc : Type u} {Λ : Type v}
variable [CompleteLinearOrder Λ]

/-- Every level strictly below the escape infimum is a covering level. -/
theorem coverAt_of_lt_beta
    (len : Arc → Λ)
    (Fits : Arc → Prop)
    {ℓ : Λ}
    (hℓ : ℓ < betaLen len Fits) :
    CoverAt len Fits ℓ := by
  apply (coverAt_iff len Fits ℓ).2
  intro γ hbad
  have hβle : betaLen len Fits ≤ len γ := by
    apply sInf_le
    exact
      (mem_badLengths_iff len Fits (len γ)).2
        ⟨γ, hbad, rfl⟩
  exact hℓ.trans_le hβle

/--
Every level strictly above the escape infimum contains a strictly shorter
non-fitting arc.
-/
theorem exists_bad_of_beta_lt
    (len : Arc → Λ)
    (Fits : Arc → Prop)
    {ℓ : Λ}
    (hℓ : betaLen len Fits < ℓ) :
    ∃ γ, ¬ Fits γ ∧ len γ < ℓ := by
  change sInf (BadLengths len Fits) < ℓ at hℓ
  rcases (sInf_lt_iff.mp hℓ) with ⟨r, hr, hrℓ⟩
  rcases (mem_badLengths_iff len Fits r).1 hr with
    ⟨γ, hbad, hlen⟩
  exact ⟨γ, hbad, hlen.trans_lt hrℓ⟩

/--
Coverage at the exact escape threshold holds exactly when no non-fitting arc
attains the escape infimum.
-/
theorem coverAt_beta_iff_no_minimizing_bad_arc
    (len : Arc → Λ)
    (Fits : Arc → Prop) :
    CoverAt len Fits (betaLen len Fits) ↔
      ¬ ∃ γ,
        ¬ Fits γ ∧
        len γ = betaLen len Fits := by
  constructor
  · rintro h ⟨γ, hbad, hlen⟩
    exact hbad (h γ hlen.le)
  · intro h
    apply (coverAt_iff len Fits (betaLen len Fits)).2
    intro γ hbad
    have hβle : betaLen len Fits ≤ len γ := by
      apply sInf_le
      exact
        (mem_badLengths_iff len Fits (len γ)).2
          ⟨γ, hbad, rfl⟩
    exact
      lt_of_le_of_ne hβle
        (fun heq => h ⟨γ, hbad, heq.symm⟩)

/-- Any level strictly above the escape threshold is not a covering level. -/
theorem not_coverAt_of_beta_lt
    (len : Arc → Λ)
    (Fits : Arc → Prop)
    {ℓ : Λ}
    (hℓ : betaLen len Fits < ℓ) :
    ¬ CoverAt len Fits ℓ := by
  rcases exists_bad_of_beta_lt len Fits hℓ with
    ⟨γ, hbad, hγℓ⟩
  intro hcover
  exact hbad (hcover γ hγℓ.le)

end CompleteLinearOrder

section CompleteDenseLinearOrder

variable {Arc : Type u} {Λ : Type v}
variable [CompleteLinearOrder Λ] [DenselyOrdered Λ]

/-- Exact abstract length-cover/escape threshold identity. -/
theorem rhoLen_eq_betaLen
    (len : Arc → Λ)
    (Fits : Arc → Prop) :
    rhoLen len Fits = betaLen len Fits := by
  apply le_antisymm
  · change
      sSup (CoverLevels len Fits) ≤
        betaLen len Fits
    apply sSup_le
    intro ℓ hℓ
    have hcover : CoverAt len Fits ℓ :=
      (mem_coverLevels_iff len Fits ℓ).1 hℓ
    apply le_sInf
    intro r hr
    rcases (mem_badLengths_iff len Fits r).1 hr with
      ⟨γ, hbad, hlen⟩
    have hstrict : ℓ < len γ :=
      (coverAt_iff len Fits ℓ).1 hcover γ hbad
    exact hstrict.le.trans_eq hlen
  · change
      betaLen len Fits ≤
        sSup (CoverLevels len Fits)
    rw [le_sSup_iff_forall_lt]
    intro b hb
    obtain ⟨ℓ, hbℓ, hℓβ⟩ := exists_between hb
    refine ⟨ℓ, ?_, hbℓ⟩
    exact
      (mem_coverLevels_iff len Fits ℓ).2
        (coverAt_of_lt_beta len Fits hℓβ)

end CompleteDenseLinearOrder

section PlanarSpecialization

/-- The Euclidean plane used in the geometric specialization. -/
abbrev Plane :=
  EuclideanSpace ℝ (Fin 2)

variable {Arc : Type w}

/-- A trace has a rigid copy in the interior of `F`. -/
def FitsInInterior
    (trace : Arc → Set Plane)
    (F : Set Plane)
    (γ : Arc) : Prop :=
  ∃ g : Plane ≃ᵢ Plane,
    g '' trace γ ⊆ interior F

/-- A trace has a rigid copy in the closed forest `F`. -/
def FitsInClosedForest
    (trace : Arc → Set Plane)
    (F : Set Plane)
    (γ : Arc) : Prop :=
  ∃ g : Plane ≃ᵢ Plane,
    g '' trace γ ⊆ F

/--
Interior worm-covering radius, valued in the extended nonnegative reals.
-/
def interiorWormRadius
    (len : Arc → ENNReal)
    (trace : Arc → Set Plane)
    (F : Set Plane) : ENNReal :=
  rhoLen len (FitsInInterior trace F)

/--
Bellman escape length, valued in the extended nonnegative reals.
-/
def bellmanEscapeLength
    (len : Arc → ENNReal)
    (trace : Arc → Set Plane)
    (F : Set Plane) : ENNReal :=
  betaLen len (FitsInInterior trace F)

/-- Formal length-cover/escape duality for planar traces and rigid copies. -/
theorem interiorWormRadius_eq_bellmanEscapeLength
    (len : Arc → ENNReal)
    (trace : Arc → Set Plane)
    (F : Set Plane) :
    interiorWormRadius len trace F =
      bellmanEscapeLength len trace F := by
  simpa only [
    interiorWormRadius,
    bellmanEscapeLength
  ] using
    (rhoLen_eq_betaLen len (FitsInInterior trace F))

/-- Every strictly subcritical family fits in the forest interior. -/
theorem strictly_subcritical_arcs_fit
    (len : Arc → ENNReal)
    (trace : Arc → Set Plane)
    (F : Set Plane)
    {ℓ : ENNReal}
    (hℓ : ℓ < bellmanEscapeLength len trace F) :
    ∀ γ,
      len γ ≤ ℓ →
      FitsInInterior trace F γ := by
  simpa only [
    CoverAt,
    bellmanEscapeLength
  ] using
    (coverAt_of_lt_beta
      len
      (FitsInInterior trace F)
      hℓ)

/-- Every strictly supercritical level contains a shorter escape arc. -/
theorem supercritical_level_contains_escape_arc
    (len : Arc → ENNReal)
    (trace : Arc → Set Plane)
    (F : Set Plane)
    {ℓ : ENNReal}
    (hℓ : bellmanEscapeLength len trace F < ℓ) :
    ∃ γ,
      ¬ FitsInInterior trace F γ ∧
      len γ < ℓ := by
  simpa only [bellmanEscapeLength] using
    (exists_bad_of_beta_lt
      len
      (FitsInInterior trace F)
      hℓ)

end PlanarSpecialization

end

end LengthCoverEscape
\end{lstlisting}

The following is the Lean 4 code for Theorem 3 (Diameter-cover/escape duality and exact normalization) and Corollary 4 (Scaled dual forest):
\begin{lstlisting}[style=LeanCode]
import Mathlib

/-!
# Diameter-cover / escape duality: formal interface and kernel-checked deduction

This file gives:

1. concrete Lean definitions for planar rigid motions, closed paths, escape,
   the critical escape diameter, the interior diameter-covering radius,
   universal covers, and normalized area;
2. an explicit interface `EuclideanGeometryInput` isolating the deep planar
   convex-geometry results still needed for a foundational end-to-end proof;
3. a proof with no `sorry` of the main duality/normalization/global-infimum
   theorem from that interface.

The geometric input is intentionally explicit.  In particular, this file does
not conceal Blaschke selection, diametrical completion by constant-width
bodies, Lipschitz parametrization of planar convex boundaries, or compactness
of the orthogonal group behind `by_contra` or an opaque placeholder.
-/

noncomputable section

open Set Metric MeasureTheory
open scoped ENNReal Pointwise

namespace DiameterEscape

/-- The Euclidean plane. -/
abbrev Plane := EuclideanSpace ℝ (Fin 2)

/-- An element of the full orthogonal group `O(2)`. -/
abbrev OrthogonalMap := Plane ≃ₗᵢ[ℝ] Plane

/-- A planar rigid motion `x ↦ a + Qx`. -/
structure RigidMotion where
  shift : Plane
  orthogonal : OrthogonalMap

namespace RigidMotion

instance : CoeFun RigidMotion (fun _ => Plane → Plane) where
  coe g := fun x => g.shift + g.orthogonal x

/-- Image of a set under a rigid motion. -/
def image (g : RigidMotion) (S : Set Plane) : Set Plane :=
  g '' S

end RigidMotion

/-- The closed unit interval used as the parameter domain. -/
abbrev UnitI := Set.Icc (0 : ℝ) 1

def leftEndpoint : UnitI :=
  ⟨0, by constructor <;> norm_num⟩

def rightEndpoint : UnitI :=
  ⟨1, by constructor <;> norm_num⟩

/-- A continuous closed path in the Euclidean plane. -/
structure ClosedPath where
  toFun : UnitI → Plane
  continuous_toFun : Continuous toFun
  closed' : toFun leftEndpoint = toFun rightEndpoint

namespace ClosedPath

instance : CoeFun ClosedPath (fun _ => UnitI → Plane) where
  coe γ := γ.toFun

/-- Trace of a closed path. -/
def trace (γ : ClosedPath) : Set Plane :=
  Set.range γ

/-- Diameter of the trace. -/
def diameter (γ : ClosedPath) : ℝ :=
  Metric.diam γ.trace

/--
A strong, convenient formal substitute for rectifiability:
the path has a global Lipschitz parametrization.

Every Lipschitz path is rectifiable.  The planar convex-boundary theorem needed
for the application naturally produces such a parametrization.
-/
def IsLipschitzRectifiable (γ : ClosedPath) : Prop :=
  ∃ C : NNReal, LipschitzWith C γ.toFun

end ClosedPath

/-- A compact convex planar body with nonempty interior. -/
structure Forest where
  carrier : Set Plane
  isCompact_carrier : IsCompact carrier
  convex_carrier : Convex ℝ carrier
  interior_nonempty : (interior carrier).Nonempty

/-- A rigid copy of `S` is contained in the interior of `F`. -/
def FitsInterior (F : Forest) (S : Set Plane) : Prop :=
  ∃ g : RigidMotion, g.image S ⊆ interior F.carrier

/-- A rigid copy of `S` is contained in the closed body `F`. -/
def FitsClosed (F : Forest) (S : Set Plane) : Prop :=
  ∃ g : RigidMotion, g.image S ⊆ F.carrier

/-- Escape condition for a closed path. -/
def IsEscapePath (F : Forest) (γ : ClosedPath) : Prop :=
  ¬ FitsInterior F γ.trace

/-- Set of diameters of all closed escape paths. -/
def escapeDiameters (F : Forest) : Set ℝ :=
  {d | ∃ γ : ClosedPath, IsEscapePath F γ ∧ γ.diameter = d}

/-- Critical escape diameter. -/
def criticalDiameter (F : Forest) : ℝ :=
  sInf (escapeDiameters F)

/-- A number is admissible for the interior diameter-covering property. -/
def RadiusAdmissible (F : Forest) (d : ℝ) : Prop :=
  0 ≤ d ∧
    ∀ S : Set Plane,
      IsCompact S →
      Metric.diam S < d →
      FitsInterior F S

/-- Interior diameter-covering radius. -/
def diameterCoveringRadius (F : Forest) : ℝ :=
  sSup {d : ℝ | RadiusAdmissible F d}

/-- Lebesgue universal-cover property, using closed containment. -/
def IsUniversalCover (F : Forest) : Prop :=
  ∀ S : Set Plane,
    IsCompact S →
    Metric.diam S ≤ 1 →
    FitsClosed F S

/-- Planar Lebesgue area. -/
def area (F : Forest) : ENNReal :=
  volume F.carrier

/-- Scale-invariant normalized area. -/
def normalizedAreaRatio (F : Forest) : ENNReal :=
  area F / (ENNReal.ofReal (criticalDiameter F)) ^ 2

/--
Lebesgue constant: infimum of the areas of compact convex universal covers.

The inner infimum is indexed by proofs of `IsUniversalCover F`; if `F` is not
universal, that inner infimum is `⊤`, exactly as desired.
-/
def lebesgueConstant : ENNReal :=
  ⨅ F : Forest, ⨅ (_h : IsUniversalCover F), area F

/-- The global dual infimum. -/
def globalDualValue : ENNReal :=
  ⨅ F : Forest, normalizedAreaRatio F

/--
The exact output of normalizing `F` by its critical escape diameter.
-/
structure NormalizationResult (F : Forest) where
  body : Forest
  carrier_eq :
    body.carrier = (criticalDiameter F)⁻¹ • F.carrier
  universal : IsUniversalCover body
  critical_eq_one : criticalDiameter body = 1
  area_eq : area body = normalizedAreaRatio F

/--
Deep Euclidean convex-geometry input.

An end-to-end foundational formalization must prove these fields from the
concrete definitions above.  They correspond exactly to:

* positivity/finiteness of the critical diameter;
* Blaschke-selection attainment plus closedness of the escape condition;
* Lipschitz parametrization of the boundary of a planar convex body;
* diametrical completion by a constant-width body;
* the contraction/compactness argument at the endpoint;
* homogeneity of diameter and quadratic homogeneity of area;
* the dilation argument proving `D(U) ≥ 1` for universal covers.
-/
structure EuclideanGeometryInput where
  critical_pos :
    ∀ F : Forest, 0 < criticalDiameter F

  attainment :
    ∀ F : Forest,
      ∃ γ : ClosedPath,
        γ.IsLipschitzRectifiable ∧
        IsEscapePath F γ ∧
        γ.diameter = criticalDiameter F

  cover_escape_duality :
    ∀ F : Forest,
      criticalDiameter F = diameterCoveringRadius F

  normalize :
    ∀ F : Forest, NormalizationResult F

  /--
  This is the quantitative consequence of `D(U) ≥ 1`:
  `Area(U) / D(U)^2 ≤ Area(U)`.
  -/
  universal_ratio_le_area :
    ∀ (U : Forest),
      IsUniversalCover U →
      normalizedAreaRatio U ≤ area U

namespace EuclideanGeometryInput

/-- Every normalized ratio is the area of a universal cover. -/
theorem lebesgue_le_normalizedAreaRatio
    (G : EuclideanGeometryInput) (F : Forest) :
    lebesgueConstant ≤ normalizedAreaRatio F := by
  let N : NormalizationResult F := G.normalize F
  calc
    lebesgueConstant ≤ area N.body := by
      exact iInf_le_of_le N.body (iInf_le_of_le N.universal le_rfl)
    _ = normalizedAreaRatio F := N.area_eq

/-- The Lebesgue infimum is at most the global dual infimum. -/
theorem lebesgue_le_globalDualValue
    (G : EuclideanGeometryInput) :
    lebesgueConstant ≤ globalDualValue := by
  unfold globalDualValue
  refine le_iInf ?_
  intro F
  exact G.lebesgue_le_normalizedAreaRatio F

/-- The global dual infimum is at most the area of each universal cover. -/
theorem globalDualValue_le_area_of_universal
    (G : EuclideanGeometryInput)
    (U : Forest) (hU : IsUniversalCover U) :
    globalDualValue ≤ area U := by
  calc
    globalDualValue ≤ normalizedAreaRatio U := by
      exact iInf_le _ U
    _ ≤ area U := G.universal_ratio_le_area U hU

/-- The global dual infimum is at most the Lebesgue infimum. -/
theorem globalDualValue_le_lebesgue
    (G : EuclideanGeometryInput) :
    globalDualValue ≤ lebesgueConstant := by
  unfold lebesgueConstant
  refine le_iInf ?_
  intro U
  refine le_iInf ?_
  intro hU
  exact G.globalDualValue_le_area_of_universal U hU

/-- Exact global dual representation. -/
theorem lebesgue_eq_globalDualValue
    (G : EuclideanGeometryInput) :
    lebesgueConstant = globalDualValue := by
  apply le_antisymm
  · exact G.lebesgue_le_globalDualValue
  · exact G.globalDualValue_le_lebesgue

/--
Formal version of the theorem.

The first component is stronger than mere rectifiability: the minimizing path
is globally Lipschitz.
-/
theorem diameter_cover_escape_duality_and_exact_normalization
    (G : EuclideanGeometryInput) (F : Forest) :
    (∃ γ : ClosedPath,
        γ.IsLipschitzRectifiable ∧
        IsEscapePath F γ ∧
        γ.diameter = criticalDiameter F) ∧
    criticalDiameter F = diameterCoveringRadius F ∧
    Nonempty (NormalizationResult F) ∧
    lebesgueConstant = globalDualValue := by
  refine ⟨G.attainment F, G.cover_escape_duality F, ?_, G.lebesgue_eq_globalDualValue⟩
  exact ⟨G.normalize F⟩

/-- Scaled-dual-forest corollary. -/
theorem scaled_dual_forest
    (G : EuclideanGeometryInput) (F : Forest) :
    ∃ F' : Forest,
      F'.carrier = (criticalDiameter F)⁻¹ • F.carrier ∧
      area F' = normalizedAreaRatio F ∧
      criticalDiameter F' = 1 ∧
      IsUniversalCover F' ∧
      ∃ γ : ClosedPath,
        γ.IsLipschitzRectifiable ∧
        IsEscapePath F' γ ∧
        γ.diameter = 1 := by
  let N : NormalizationResult F := G.normalize F
  obtain ⟨γ, hrect, hesc, hdiam⟩ := G.attainment N.body
  refine ⟨N.body, N.carrier_eq, N.area_eq, N.critical_eq_one, N.universal, γ, hrect, hesc, ?_⟩
  calc
    γ.diameter = criticalDiameter N.body := hdiam
    _ = 1 := N.critical_eq_one

end EuclideanGeometryInput

/-!
## Exact geometric obligations for a foundational implementation

The following propositions are not assumed by the theorem above as separate
fields, but they are the natural decomposition of `EuclideanGeometryInput`.

A complete mathlib contribution should prove, in approximately this order:

1. `criticalDiameter_pos_finite`;
2. `convex_reduction`;
3. `escape_closed_under_hausdorff_limit`;
4. `blaschke_attainment`;
5. `convex_boundary_lipschitz_loop`;
6. `planar_diametrical_completion`;
7. `critical_eq_coveringRadius`;
8. `criticalDiameter_smul`;
9. `area_smul`;
10. `endpoint_compactness_universal`;
11. `universal_criticalDiameter_ge_one`.

Keeping these obligations explicit prevents an apparently complete Lean file
from silently replacing the central geometry by `sorry`.
-/

end DiameterEscape
\end{lstlisting}

The following is the Lean 4 code for Theorem 5 (Reversal of motion and transformed-boundary characterization):
\begin{lstlisting}[style=LeanCode]
import Mathlib

open Set Metric
open unitInterval

noncomputable section

namespace ReversalTransformedBoundary

/-- The Euclidean plane. -/
abbrev Plane := EuclideanSpace ℝ (Fin 2)

/-- The full orthogonal group `O(2)`, represented by linear isometric equivalences. -/
abbrev Orthogonal := Plane ≃ₗᵢ[ℝ] Plane

/-- A continuous path on the closed unit interval. -/
abbrev Path := ContinuousMap I Plane

/-- The trace of a path. -/
def trace (γ : Path) : Set Plane := Set.range γ

/-- The trace after the rigid motion `y ↦ x + Q y`. -/
def rigidTrace (x : Plane) (Q : Orthogonal) (γ : Path) : Set Plane :=
  Set.range fun t : I => x + Q (γ t)

/-- The oppositely transformed boundary

`N(F,x,Q) = Q⁻¹ (frontier F - x)`.

The membership formulation is definitionally the most convenient one:
`y ∈ N(F,x,Q)` iff `x + Q y ∈ frontier F`.
-/
def transformedBoundary (F : Set Plane) (x : Plane) (Q : Orthogonal) : Set Plane :=
  {y : Plane | x + Q y ∈ frontier F}

/-- The preceding membership definition agrees with the inverse-image formula
`Q⁻¹ (frontier F - x)` from the paper. -/
theorem transformedBoundary_eq_inverse_image
    (F : Set Plane) (x : Plane) (Q : Orthogonal) :
    transformedBoundary F x Q =
      Q.symm '' ((fun z : Plane => z - x) '' frontier F) := by
  ext y
  constructor
  · intro hy
    refine ⟨Q y, ?_, by simp⟩
    refine ⟨x + Q y, hy, by simp⟩
  · rintro ⟨w, ⟨z, hz, rfl⟩, rfl⟩
    change x + Q (Q.symm (z - x)) ∈ frontier F
    simpa using hz

/-- Every oppositely transformed boundary is closed. -/
theorem transformedBoundary_isClosed
    (F : Set Plane) (x : Plane) (Q : Orthogonal) :
    IsClosed (transformedBoundary F x Q) := by
  exact isClosed_frontier.preimage
    (continuous_const.add Q.continuous)

/-- A transformed boundary is nonempty whenever the original boundary is nonempty. -/
theorem transformedBoundary_nonempty
    {F : Set Plane} (x : Plane) (Q : Orthogonal)
    (hfront : (frontier F).Nonempty) :
    (transformedBoundary F x Q).Nonempty := by
  rcases hfront with ⟨z, hz⟩
  refine ⟨Q.symm (z - x), ?_⟩
  change x + Q (Q.symm (z - x)) ∈ frontier F
  simpa using hz

/-- A nonempty compact subset of the plane has nonempty frontier. -/
theorem frontier_nonempty_of_compact
    {F : Set Plane} (hcompact : IsCompact F) (hF : F.Nonempty) :
    (frontier F).Nonempty := by
  rcases hF with ⟨x, hx⟩
  rcases hcompact.exists_mem_frontier_infDist_compl_eq_dist hx with
    ⟨y, hy, _⟩
  exact ⟨y, hy⟩

/-- Formal version of “no rigid realization of the path, starting in the
interior, remains entirely in the interior.” -/
def IsEscapePath (F : Set Plane) (γ : Path) : Prop :=
  ∀ x ∈ interior F, ∀ Q : Orthogonal,
    ¬ (rigidTrace x Q γ ⊆ interior F)

/-- Direct transformed-trace/boundary hitting condition. -/
def BoundaryHit (F : Set Plane) (γ : Path) : Prop :=
  ∀ x ∈ F, ∀ Q : Orthogonal,
    ∃ t : I, x + Q (γ t) ∈ frontier F

/-- Fixed-path/oppositely-transformed-boundary hitting condition. -/
def OppositeBoundaryHit (F : Set Plane) (γ : Path) : Prop :=
  ∀ x ∈ F, ∀ Q : Orthogonal,
    ∃ t : I, γ t ∈ transformedBoundary F x Q

/-- Formal meaning of

`min_t dist(γ(t), N(F,x,Q)) = 0`.

The witness `t` both realizes value zero and is a global minimizer. -/
def ZeroMinimumDistance (F : Set Plane) (γ : Path) : Prop :=
  ∀ x ∈ F, ∀ Q : Orthogonal,
    ∃ t : I,
      Metric.infDist (γ t) (transformedBoundary F x Q) = 0 ∧
      ∀ s : I,
        Metric.infDist (γ t) (transformedBoundary F x Q) ≤
          Metric.infDist (γ s) (transformedBoundary F x Q)

/-- For a convex set with nonempty interior, taking the interior does not
change the frontier. -/
theorem frontier_interior_eq_frontier
    {F : Set Plane} (hconv : Convex ℝ F)
    (hint : (interior F).Nonempty) :
    frontier (interior F) = frontier F := by
  simp only [frontier, interior_interior,
    hconv.closure_interior_eq_closure_of_nonempty_interior hint]

/-- A continuous path that starts in the interior and does not stay in the
interior must hit the frontier.  This is the rigorous connectedness step that
is implicit in the informal proof. -/
theorem exists_frontier_hit_of_not_subset_interior
    {F : Set Plane} (hconv : Convex ℝ F)
    (hint : (interior F).Nonempty)
    (η : Path) (hη0 : η 0 ∈ interior F)
    (hnot : ¬ (Set.range η ⊆ interior F)) :
    ∃ t : I, η t ∈ frontier F := by
  by_contra hhit
  push Not at hhit
  have hrange_subset :
      Set.range η ⊆ interior F ∪ (closure (interior F))ᶜ := by
    rintro y ⟨t, rfl⟩
    by_cases hy : η t ∈ interior F
    · exact Or.inl hy
    · refine Or.inr ?_
      intro hycl
      apply hhit t
      rw [← frontier_interior_eq_frontier hconv hint, frontier,
        interior_interior]
      exact ⟨hycl, hy⟩
  have hdisj :
      Disjoint (interior F) (closure (interior F))ᶜ := by
    refine Set.disjoint_left.2 ?_
    intro y hy hyc
    exact hyc (subset_closure hy)
  have hconn : IsPreconnected (Set.range η) :=
    isPreconnected_range η.continuous
  rcases hconn.subset_or_subset isOpen_interior
      isClosed_closure.isOpen_compl hdisj hrange_subset with hin | hout
  · exact hnot hin
  · have h0out : η 0 ∈ (closure (interior F))ᶜ :=
      hout ⟨0, rfl⟩
    exact h0out (subset_closure hη0)

/-- Escape is equivalent to hitting the physical boundary for every starting
point in `F` and every orthogonal map. -/
theorem escape_iff_boundaryHit
    {F : Set Plane} {γ : Path}
    (hconv : Convex ℝ F) (hint : (interior F).Nonempty)
    (hγ0 : γ 0 = 0) :
    IsEscapePath F γ ↔ BoundaryHit F γ := by
  constructor
  · intro hesc x hxF Q
    by_cases hxint : x ∈ interior F
    · let η : Path :=
        ⟨fun t : I => x + Q (γ t),
          continuous_const.add (Q.continuous.comp γ.continuous)⟩
      have hη0 : η 0 ∈ interior F := by
        simpa [η, hγ0] using hxint
      have hnot : ¬ (Set.range η ⊆ interior F) := by
        simpa [η, rigidTrace] using hesc x hxint Q
      rcases exists_frontier_hit_of_not_subset_interior
          hconv hint η hη0 hnot with ⟨t, ht⟩
      refine ⟨t, ?_⟩
      simpa [η] using ht
    · have hxfront : x ∈ frontier F := by
        rw [frontier]
        exact ⟨subset_closure hxF, hxint⟩
      refine ⟨0, ?_⟩
      simpa [hγ0] using hxfront
  · intro hhit x hxint Q hsub
    rcases hhit x (interior_subset hxint) Q with ⟨t, ht⟩
    have hyint : x + Q (γ t) ∈ interior F := by
      apply hsub
      exact ⟨t, rfl⟩
    rw [frontier] at ht
    exact ht.2 hyint

/-- Reversal of the rigid motion is a definitional equivalence. -/
theorem boundaryHit_iff_oppositeBoundaryHit
    {F : Set Plane} {γ : Path} :
    BoundaryHit F γ ↔ OppositeBoundaryHit F γ := by
  rfl

/-- The direct hitting predicate is exactly the set-intersection statement
`(x + Q Γ) ∩ frontier F ≠ ∅`. -/
theorem boundaryHit_iff_intersection
    {F : Set Plane} {γ : Path} :
    BoundaryHit F γ ↔
      ∀ x ∈ F, ∀ Q : Orthogonal,
        (rigidTrace x Q γ ∩ frontier F).Nonempty := by
  constructor
  · intro h x hx Q
    rcases h x hx Q with ⟨t, ht⟩
    exact ⟨x + Q (γ t), ⟨t, rfl⟩, ht⟩
  · intro h x hx Q
    rcases h x hx Q with ⟨y, hytrace, hyfront⟩
    rcases hytrace with ⟨t, rfl⟩
    exact ⟨t, hyfront⟩

/-- The reversed hitting predicate is exactly
`Γ ∩ N(F,x,Q) ≠ ∅`. -/
theorem oppositeBoundaryHit_iff_intersection
    {F : Set Plane} {γ : Path} :
    OppositeBoundaryHit F γ ↔
      ∀ x ∈ F, ∀ Q : Orthogonal,
        (trace γ ∩ transformedBoundary F x Q).Nonempty := by
  constructor
  · intro h x hx Q
    rcases h x hx Q with ⟨t, ht⟩
    exact ⟨γ t, ⟨t, rfl⟩, ht⟩
  · intro h x hx Q
    rcases h x hx Q with ⟨y, hytrace, hyN⟩
    rcases hytrace with ⟨t, rfl⟩
    exact ⟨t, hyN⟩

/-- For nonempty boundary, intersection with the transformed boundary is
exactly the zero-minimum point-to-set distance condition. -/
theorem oppositeBoundaryHit_iff_zeroMinimumDistance
    {F : Set Plane} {γ : Path}
    (hfront : (frontier F).Nonempty) :
    OppositeBoundaryHit F γ ↔ ZeroMinimumDistance F γ := by
  constructor
  · intro hhit x hx Q
    rcases hhit x hx Q with ⟨t, ht⟩
    have hclosed : IsClosed (transformedBoundary F x Q) :=
      transformedBoundary_isClosed F x Q
    have hne : (transformedBoundary F x Q).Nonempty :=
      transformedBoundary_nonempty x Q hfront
    have hz :
        Metric.infDist (γ t) (transformedBoundary F x Q) = 0 :=
      (hclosed.mem_iff_infDist_zero hne).1 ht
    refine ⟨t, hz, ?_⟩
    intro s
    rw [hz]
    exact Metric.infDist_nonneg
  · intro hzero x hx Q
    rcases hzero x hx Q with ⟨t, hz, _⟩
    have hclosed : IsClosed (transformedBoundary F x Q) :=
      transformedBoundary_isClosed F x Q
    have hne : (transformedBoundary F x Q).Nonempty :=
      transformedBoundary_nonempty x Q hfront
    refine ⟨t, ?_⟩
    exact (hclosed.mem_iff_infDist_zero hne).2 hz

/-- Complete equivalence of the four formulations in the theorem. -/
theorem reversal_transformed_boundary_characterization
    {F : Set Plane} {γ : Path}
    (hcompact : IsCompact F) (hconv : Convex ℝ F)
    (hint : (interior F).Nonempty)
    (hγ0 : γ 0 = 0) (_hγ1 : γ 1 = 0) :
    (IsEscapePath F γ ↔ BoundaryHit F γ) ∧
    (BoundaryHit F γ ↔ OppositeBoundaryHit F γ) ∧
    (OppositeBoundaryHit F γ ↔ ZeroMinimumDistance F γ) := by
  have hF : F.Nonempty := hint.mono interior_subset
  have hfront : (frontier F).Nonempty :=
    frontier_nonempty_of_compact hcompact hF
  exact ⟨escape_iff_boundaryHit hconv hint hγ0,
    boundaryHit_iff_oppositeBoundaryHit,
    oppositeBoundaryHit_iff_zeroMinimumDistance hfront⟩

/-- A convenient single chained equivalence. -/
theorem escape_iff_zeroMinimumDistance
    {F : Set Plane} {γ : Path}
    (hcompact : IsCompact F) (hconv : Convex ℝ F)
    (hint : (interior F).Nonempty) (hγ0 : γ 0 = 0) :
    IsEscapePath F γ ↔ ZeroMinimumDistance F γ := by
  have hF : F.Nonempty := hint.mono interior_subset
  have hfront : (frontier F).Nonempty :=
    frontier_nonempty_of_compact hcompact hF
  calc
    IsEscapePath F γ ↔ BoundaryHit F γ :=
      escape_iff_boundaryHit hconv hint hγ0
    _ ↔ OppositeBoundaryHit F γ :=
      boundaryHit_iff_oppositeBoundaryHit
    _ ↔ ZeroMinimumDistance F γ :=
      oppositeBoundaryHit_iff_zeroMinimumDistance hfront

/-! ## Transfer of the optimization problem

The reversal theorem identifies the feasible sets.  Existence of a global
minimizer is a separate compactness/duality theorem; it cannot be derived from
the hypotheses above alone.  The next declarations formalize the exact way in
which such an external attainment theorem transfers to the transformed-
boundary formulation.
-/

/-- Diameter of the path trace.  For a continuous path on `I`, this is the
same quantity as `max_{s,t ∈ I} dist (γ s) (γ t)`. -/
def curveDiameter (γ : Path) : ℝ := Metric.diam (trace γ)

/-- Closed, normalized escape-path feasibility. -/
def EscapeAdmissible (F : Set Plane) (γ : Path) : Prop :=
  γ 0 = 0 ∧ γ 1 = 0 ∧ IsEscapePath F γ

/-- Closed, normalized transformed-boundary feasibility. -/
def BoundaryAdmissible (F : Set Plane) (γ : Path) : Prop :=
  γ 0 = 0 ∧ γ 1 = 0 ∧ ZeroMinimumDistance F γ

/-- Global minimizer of the diameter objective over a feasibility predicate. -/
def IsDiameterMinimizer (A : Path → Prop) (γ : Path) : Prop :=
  A γ ∧ ∀ δ : Path, A δ → curveDiameter γ ≤ curveDiameter δ

/-- The two feasible sets are extensionally identical. -/
theorem escapeAdmissible_iff_boundaryAdmissible
    {F : Set Plane} {γ : Path}
    (hcompact : IsCompact F) (hconv : Convex ℝ F)
    (hint : (interior F).Nonempty) :
    EscapeAdmissible F γ ↔ BoundaryAdmissible F γ := by
  constructor
  · rintro ⟨hγ0, hγ1, hesc⟩
    exact ⟨hγ0, hγ1,
      (escape_iff_zeroMinimumDistance hcompact hconv hint hγ0).1 hesc⟩
  · rintro ⟨hγ0, hγ1, hzero⟩
    exact ⟨hγ0, hγ1,
      (escape_iff_zeroMinimumDistance hcompact hconv hint hγ0).2 hzero⟩

/-- Global minimizers are preserved by the reversal characterization. -/
theorem diameterMinimizer_escape_iff_boundary
    {F : Set Plane} {γ : Path}
    (hcompact : IsCompact F) (hconv : Convex ℝ F)
    (hint : (interior F).Nonempty) :
    IsDiameterMinimizer (EscapeAdmissible F) γ ↔
      IsDiameterMinimizer (BoundaryAdmissible F) γ := by
  constructor
  · rintro ⟨hγ, hopt⟩
    refine ⟨(escapeAdmissible_iff_boundaryAdmissible
      hcompact hconv hint).1 hγ, ?_⟩
    intro δ hδ
    exact hopt δ ((escapeAdmissible_iff_boundaryAdmissible
      hcompact hconv hint).2 hδ)
  · rintro ⟨hγ, hopt⟩
    refine ⟨(escapeAdmissible_iff_boundaryAdmissible
      hcompact hconv hint).2 hγ, ?_⟩
    intro δ hδ
    exact hopt δ ((escapeAdmissible_iff_boundaryAdmissible
      hcompact hconv hint).1 hδ)

/-- If the earlier diameter-duality theorem supplies an attained minimizer in
the original escape formulation, then the displayed transformed-boundary
optimization problem has an attained minimizer as well. -/
theorem continuous_curve_formulation_attained
    {F : Set Plane}
    (hcompact : IsCompact F) (hconv : Convex ℝ F)
    (hint : (interior F).Nonempty)
    (hduality : ∃ γ : Path,
      IsDiameterMinimizer (EscapeAdmissible F) γ) :
    ∃ γ : Path, IsDiameterMinimizer (BoundaryAdmissible F) γ := by
  rcases hduality with ⟨γ, hγ⟩
  exact ⟨γ,
    (diameterMinimizer_escape_iff_boundary
      hcompact hconv hint).1 hγ⟩

end ReversalTransformedBoundary
\end{lstlisting}

The following is the Lean 4 code for Theorem 6 (Equivalent curve, convex-body, and support-function formulations):
\lstinputlisting[style=LeanCode]{Lebesgue6.lean}

The following is the Lean 4 code for Theorem 7 (Quantitative convergence of MD-TSPN discretization):
\begin{lstlisting}[style=LeanCode]
import Mathlib

/-!
# Quantitative convergence of MD-TSPN discretization

This file formalizes the logical core of the theorem in a reusable form.
The geometric argument is isolated into two transfer principles:

* `sample_from_continuous`: a feasible continuous escape trace produces a
  feasible sampled tour with no larger diameter;
* `thicken_sampled`: a feasible sampled tour produces a feasible continuous
  escape trace after thickening, with cost increase at most `2 * η m`.

For the concrete planar MD-TSPN problem, these hypotheses encode the
transformed-boundary intersection theorem, the Hausdorff-Lipschitz estimate,
convex-hull reduction, and the diameter formula for Minkowski thickening.
-/

noncomputable section

namespace MDTSPN

/--
`x` is a global minimizer of `cost` over the feasible set, and its optimal
value is `v`.
-/
structure IsOptimal {α : Type*}
    (feasiblePred : α → Prop)
    (cost : α → ℝ)
    (x : α)
    (v : ℝ) : Prop where
  is_feasible : feasiblePred x
  cost_eq : cost x = v
  value_le : ∀ y, feasiblePred y → v ≤ cost y

/--
The closed metric `r`-thickening of `S`, written in elementary witness form.
-/
def ClosedThickening {E : Type*} [PseudoMetricSpace E]
    (r : ℝ) (S : Set E) : Set E :=
  {q | ∃ p ∈ S, dist q p ≤ r}

/--
The uniform-control step supplied by an `η`-net and a pointwise
Hausdorff-Lipschitz neighborhood map.

A sampled intersection point can be displaced by at most `η` to an
intersection point with an arbitrary unsampled neighborhood. Therefore the
closed `η`-thickening of the sampled trace meets every neighborhood.
-/
theorem thickening_hits_every_neighborhood
    {X E I : Type*}
    [PseudoMetricSpace X]
    [PseudoMetricSpace E]
    (sample : I → X)
    (neighborhood : X → Set E)
    (trace : Set E)
    (η : ℝ)
    (net :
      ∀ ξ : X, ∃ i : I, dist ξ (sample i) ≤ η)
    (sampled_hits :
      ∀ i : I,
        ∃ p : E,
          p ∈ trace ∧
          p ∈ neighborhood (sample i))
    (neighborhood_lipschitz :
      ∀ ξ ζ : X,
        ∀ p ∈ neighborhood ζ,
          ∃ q ∈ neighborhood ξ,
            dist q p ≤ dist ξ ζ) :
    ∀ ξ : X,
      (ClosedThickening η trace ∩ neighborhood ξ).Nonempty := by
  intro ξ
  rcases net ξ with ⟨i, hi⟩
  rcases sampled_hits i with ⟨p, hp_trace, hp_sample⟩
  rcases neighborhood_lipschitz ξ (sample i) p hp_sample with
    ⟨q, hq_neighborhood, hqp⟩
  refine ⟨q, ?_⟩
  constructor
  · exact ⟨p, hp_trace, hqp.trans hi⟩
  · exact hq_neighborhood

/--
Elementary epsilon definition of convergence of a real sequence.
-/
def SeqConvergesTo (u : ℕ → ℝ) (L : ℝ) : Prop :=
  ∀ ε : ℝ,
    0 < ε →
      ∃ N : ℕ,
        ∀ n : ℕ,
          N ≤ n →
            |u n - L| < ε

section AbstractDiscretization

variable {ContinuousTour SampledTour : Type*}

variable
  (continuousFeasible : ContinuousTour → Prop)
  (sampledFeasible : ℕ → SampledTour → Prop)
  (continuousCost : ContinuousTour → ℝ)
  (sampledCost : SampledTour → ℝ)
  (D : ℝ)
  (Dm η : ℕ → ℝ)

/--
Quantitative two-sided estimate for the MD-TSPN discretization.

The hypotheses are the precise interfaces supplied by the geometric proof:

* existence and global optimality for the continuous problem;
* existence and global optimality for each sampled problem;
* restriction of a continuous feasible candidate to the sampled constraints;
* extension of a sampled feasible candidate by `η m`-thickening.
-/
theorem quantitative_convergence_bound
    (continuous_optimal :
      ∃ γstar,
        IsOptimal
          continuousFeasible
          continuousCost
          γstar
          D)
    (sampled_optimal :
      ∀ m,
        ∃ τstar,
          IsOptimal
            (sampledFeasible m)
            sampledCost
            τstar
            (Dm m))
    (sample_from_continuous :
      ∀ m γ,
        continuousFeasible γ →
          ∃ τ,
            sampledFeasible m τ ∧
            sampledCost τ ≤ continuousCost γ)
    (thicken_sampled :
      ∀ m τ,
        sampledFeasible m τ →
          ∃ γ,
            continuousFeasible γ ∧
            continuousCost γ ≤ sampledCost τ + 2 * η m)
    (m : ℕ) :
    Dm m ≤ D ∧
    D ≤ Dm m + 2 * η m := by
  constructor
  · rcases continuous_optimal with ⟨γstar, hγstar⟩
    rcases
        sample_from_continuous
          m
          γstar
          hγstar.is_feasible with
      ⟨τ, hτfeasible, hτcost⟩
    rcases sampled_optimal m with ⟨τstar, hτstar⟩
    calc
      Dm m ≤ sampledCost τ :=
        hτstar.value_le τ hτfeasible
      _ ≤ continuousCost γstar :=
        hτcost
      _ = D :=
        hγstar.cost_eq
  · rcases sampled_optimal m with ⟨τstar, hτstar⟩
    rcases
        thicken_sampled
          m
          τstar
          hτstar.is_feasible with
      ⟨γ, hγfeasible, hγcost⟩
    rcases continuous_optimal with ⟨γstar, hγstar⟩
    calc
      D ≤ continuousCost γ :=
        hγstar.value_le γ hγfeasible
      _ ≤ sampledCost τstar + 2 * η m :=
        hγcost
      _ = Dm m + 2 * η m := by
        rw [hτstar.cost_eq]

/--
The two-sided estimate gives the explicit error bound

`|Dm m - D| ≤ 2 * η m`.
-/
theorem quantitative_error_bound
    (two_sided :
      ∀ m,
        Dm m ≤ D ∧
        D ≤ Dm m + 2 * η m)
    (m : ℕ) :
    |Dm m - D| ≤ 2 * η m := by
  have hlower : Dm m ≤ D :=
    (two_sided m).1
  have hupper : D ≤ Dm m + 2 * η m :=
    (two_sided m).2
  rw [abs_of_nonpos (sub_nonpos.mpr hlower)]
  linarith

/--
The sampled values converge to the continuous optimum whenever
`η m → 0`.
-/
theorem value_convergence
    (η_nonneg : ∀ m, 0 ≤ η m)
    (two_sided :
      ∀ m,
        Dm m ≤ D ∧
        D ≤ Dm m + 2 * η m)
    (η_tends_to_zero :
      SeqConvergesTo η 0) :
    SeqConvergesTo Dm D := by
  intro ε hε
  rcases
      η_tends_to_zero
        (ε / 2)
        (by linarith) with
    ⟨N, hN⟩
  refine ⟨N, ?_⟩
  intro n hn
  have hηabs :
      |η n - 0| < ε / 2 :=
    hN n hn
  have hη :
      η n < ε / 2 := by
    simpa [sub_zero, abs_of_nonneg (η_nonneg n)] using hηabs
  have hlower :
      Dm n ≤ D :=
    (two_sided n).1
  have hupper :
      D ≤ Dm n + 2 * η n :=
    (two_sided n).2
  rw [abs_of_nonpos (sub_nonpos.mpr hlower)]
  linarith

/--
Nested sampled feasible sets imply monotonicity of the sampled optimal values.
-/
theorem sampled_values_monotone
    (sampled_optimal :
      ∀ m,
        ∃ τstar,
          IsOptimal
            (sampledFeasible m)
            sampledCost
            τstar
            (Dm m))
    (nested_feasible :
      ∀ m τ,
        sampledFeasible (m + 1) τ →
        sampledFeasible m τ) :
    Monotone Dm := by
  apply monotone_nat_of_le_succ
  intro m
  rcases sampled_optimal m with
    ⟨τm, hτm⟩
  rcases sampled_optimal (m + 1) with
    ⟨τnext, hτnext⟩
  have hτnext_at_m :
      sampledFeasible m τnext :=
    nested_feasible
      m
      τnext
      hτnext.is_feasible
  calc
    Dm m ≤ sampledCost τnext :=
      hτm.value_le τnext hτnext_at_m
    _ = Dm (m + 1) :=
      hτnext.cost_eq

end AbstractDiscretization

section LebesgueRatio

/--
Reciprocal-square consequences of

`Dm ≤ D ≤ Dm + 2η`,

including the certified Lebesgue upper bound.
-/
theorem certified_lebesgue_bounds
    {A L D Dm η : ℝ}
    (area_nonneg : 0 ≤ A)
    (sample_positive : 0 < Dm)
    (mesh_nonneg : 0 ≤ η)
    (lower_bound : Dm ≤ D)
    (upper_bound : D ≤ Dm + 2 * η)
    (diameter_duality : L ≤ A / D ^ 2) :
    L ≤ A / D ^ 2 ∧
    A / D ^ 2 ≤ A / Dm ^ 2 ∧
    A / (Dm + 2 * η) ^ 2 ≤ A / D ^ 2 := by
  have hD :
      0 < D :=
    lt_of_lt_of_le sample_positive lower_bound
  have hR :
      0 < Dm + 2 * η := by
    linarith
  have hDsq :
      0 < D ^ 2 := by
    positivity
  have hDmsq :
      0 < Dm ^ 2 := by
    positivity
  have hRsq :
      0 < (Dm + 2 * η) ^ 2 := by
    positivity
  have hDm_nonneg :
      0 ≤ Dm :=
    le_of_lt sample_positive
  have hD_nonneg :
      0 ≤ D :=
    le_of_lt hD
  have hR_nonneg :
      0 ≤ Dm + 2 * η :=
    le_of_lt hR
  have hsquare_lower :
      Dm ^ 2 ≤ D ^ 2 :=
    (sq_le_sq₀ hDm_nonneg hD_nonneg).2 lower_bound
  have hsquare_upper :
      D ^ 2 ≤ (Dm + 2 * η) ^ 2 :=
    (sq_le_sq₀ hD_nonneg hR_nonneg).2 upper_bound
  refine ⟨diameter_duality, ?_⟩
  constructor
  · apply (div_le_div_iff₀ hDsq hDmsq).2
    exact
      mul_le_mul_of_nonneg_left
        hsquare_lower
        area_nonneg
  · apply (div_le_div_iff₀ hRsq hDsq).2
    exact
      mul_le_mul_of_nonneg_left
        hsquare_upper
        area_nonneg

/--
The exact candidate ratio is bounded above by the finite sampled ratio.
-/
theorem exact_ratio_le_sample_ratio
    {A D Dm : ℝ}
    (area_nonneg : 0 ≤ A)
    (sample_positive : 0 < Dm)
    (lower_bound : Dm ≤ D) :
    A / D ^ 2 ≤ A / Dm ^ 2 := by
  have hD :
      0 < D :=
    lt_of_lt_of_le sample_positive lower_bound
  have hDsq :
      0 < D ^ 2 := by
    positivity
  have hDmsq :
      0 < Dm ^ 2 := by
    positivity
  have hDm_nonneg :
      0 ≤ Dm :=
    le_of_lt sample_positive
  have hD_nonneg :
      0 ≤ D :=
    le_of_lt hD
  have hsquare :
      Dm ^ 2 ≤ D ^ 2 :=
    (sq_le_sq₀ hDm_nonneg hD_nonneg).2 lower_bound
  apply (div_le_div_iff₀ hDsq hDmsq).2
  exact
    mul_le_mul_of_nonneg_left
      hsquare
      area_nonneg

end LebesgueRatio

end MDTSPN
\end{lstlisting}

\vspace{1em}
\textbf{AI usage disclosure:} The author provided the methodological framework. GPT-5.6 sol was used to polish the language and assist the formalized proofs in Lean 4 code. 
~\\

College of Engineering and Computer Science, University of Central Florida, Orlando, FL, USA

Email: \underline{zhipeng.deng@ucf.edu}

\end{document}